\documentclass{baustmsmodified}

\pdfoutput=1 

\citesort
\usepackage[british]{babel}
\usepackage{xr}
\usepackage{cleveref}
\usepackage[shortlabels]{enumitem}
\usepackage{mathtools}
\usepackage{bm}
\usepackage{centernot}
\usepackage{slashed}

\usepackage{microtype}
\usepackage{anyfontsize}

\DeclareMathOperator{\A}{A}
\DeclareMathOperator{\D}{D}
\DeclareMathOperator{\F}{F}
\DeclareMathOperator{\R}{R}

\DeclareMathOperator{\dom}{dom}

\DeclareMathOperator{\At}{At}

\newcommand \compo {\mathbin{;}}

\newcommand \bmeet {\cdot}

\newcommand \rest {\mathbin{\vartriangleright}}
\newcommand \aand \wedge

\newcommand{\meet}{\prod}
\newcommand{\join}{\sum}

\newcommand{\from}{\colon}

\newcommand{\algebra}[1]{\mathfrak{#1}}
\newcommand{\defn}[1]{\textbf{#1}}

\makeatletter
\newcommand\mynobreakpar{\par\nobreak\@afterheading}
\makeatother

\theoremstyle{cupthm}
\newtheorem{theorem}{Theorem}[section]
\newtheorem{proposition}[theorem]{Proposition}
\newtheorem{corollary}[theorem]{Corollary}
\newtheorem{lemma}[theorem]{Lemma}

\theoremstyle{cupdefn}
\newtheorem{definition}[theorem]{Definition}

\theoremstyle{cuprem}

\newtheorem{problem}[theorem]{Problem}
\newtheorem{example}[theorem]{Example}

\numberwithin{equation}{section}
\usepackage{tikz}
\usetikzlibrary{arrows}
\usepackage{bookmark}
\usepackage{setspace}\onehalfspace
\newcommand{\parrow}{\rightharpoonup}

\newcommand{\typeface}{\mathbf}
\newcommand{\setq}{\typeface{
    Set_q}}
\newcommand{\aralg}{\typeface{AtRepAlg}}
\newcommand{\cA}{\algebra A} \newcommand{\cB}{\algebra B}
\newcommand{\cC}{\algebra C} 
\newcommand{\cP}{\algebra P} \newcommand{\cQ}{\algebra Q}

\newcommand*{\textrel}[1]{\mathrel{\textnormal{#1}}}

\begin{document}
\runningtitle{Difference--restriction algebras of partial functions with operators: discrete duality and completion}
\title{Difference--restriction algebras of partial functions with operators: discrete duality and completion}
\author[1]{C\'elia Borlido$^{\,1}$}
\author[2]{Brett McLean$^{*, \,2}$}

\authorheadline{C\'elia Borlido and Brett McLean}

\support{ * Corresponding author.

  \noindent
  1.  CMUC, Departamento de Matem\'atica, Universidade de Coimbra,
  \mbox{3001-501}~Coimbra, Portugal
  \email{cborlido@mat.uc.pt},  \qquad
  ORCID: 0000-0002-0114-1572
  
  \noindent
  2. Laboratoire J. A. Dieudonn\'e UMR CNRS 7351, Universit\'e Nice
  Sophia Antipolis, 06108~Nice~Cedex~02,
  France\email{brett.mclean@unice.fr},  \qquad ORCID: 0000-0003-2368-8357}


\begin{abstract}
We exhibit an adjunction between a category of abstract algebras of partial functions and a category of set quotients. The algebras are those atomic algebras representable as a collection of partial functions closed under relative complement and domain restriction; the morphisms are the complete homomorphisms. This generalises the discrete adjunction between the atomic Boolean algebras and the category of sets. We define the compatible completion of a representable algebra, and show that the monad induced by our adjunction yields the compatible completion of any atomic representable algebra. As a corollary, the adjunction restricts to a duality on the compatibly complete atomic representable algebras, generalising the discrete duality between complete atomic Boolean algebras and sets. We then extend these adjunction, duality, and completion results to representable algebras equipped with arbitrary additional completely additive and compatibility preserving operators.
\end{abstract}

\maketitle 

\section*{Declarations}

\subsection*{Funding}
The first author was partially supported by the Centre for Mathematics
of the University of Coimbra - UIDB/00324/2020, funded by the
Portuguese Government through FCT/MCTES and partially supported by
the European Research Council (ERC) under the European Union's Horizon
2020 research and innovation program (grant agreement No. 670624).
The second author was partially supported by
the European Research Council (ERC) under the European Union's Horizon
2020 research and innovation program (grant agreement No. 670624) and partially supported by the Research Foundation -- Flanders (FWO) under the SNSF--FWO Lead Agency Grant 200021L 196176 (SNSF)/G0E2121N (FWO).

\section{Introduction}

The study of algebras of partial functions is an active area of research that investigates collections of partial functions and their interrelationships from an algebraic perspective. The partial functions are treated as abstract elements that may be combined algebraically using various natural operations such as composition, domain restriction, `override', or `update'. In pure mathematics, algebras of partial functions arise naturally as structures such as inverse semigroups~\cite{wagnergeneralised}, pseudogroups~\cite{LAWSON2013117}, and skew lattices~\cite{Leech19967}. In theoretical computer science, they appear in the theories of finite state transducers \cite{10.1145/2984450.2984453},  computable functions \cite{JACKSON2015259}, deterministic propositional dynamic logics \cite{DBLP:journals/ijac/JacksonS11}, and separation logic \cite{disjoint}. 
Many different selections of operations have been considered, each leading to a different class/category of abstract algebras \cite{schein, garvacii71, Tro73, schein1992difference, 1018.20057, 1182.20058, DBLP:journals/ijac/JacksonS11, hirsch, BERENDSEN2010141, JACKSON2021106532}. (See \cite[\S 3.2]{phdthesis} for a guide to this literature.) Recently, dualities for some of these categories have started to appear \cite{lawson_2010, lawson2012non, kudryavtseva2017perspective, LAWSON201677, LAWSON2013117, 2009.07895, Bauer_2013, kudryavtseva2016boolean}, opening the way for these algebras to be studied via their duals, as has been done successfully for the algebraisations of classical and many non-classical propositional logics \cite{goldblatt, esakia, 10.2307/27588391, GEHRKE2014290}.

In \cite{diff-rest1}, we initiated a project to develop a general and modular framework for producing and understanding dualities for such categories. For this we are inspired strongly by J\'onsson and Tarski's theory of Boolean algebras with operators \cite{1951} and the duality between them and descriptive general frames \cite[Chapter 5: Algebras and General Frames]{blackburn_rijke_venema_2001}. Our central thesis is that in our case the appropriate base class---the analogue of Boolean algebras---must be more than just a class of \emph{ordered} structures but must record additional \emph{compatibility} data. This reflects the fact that the union of two partial functions is not always a function (and this determination can not be made solely from the inclusion/extension ordering).

In \cite{diff-rest1}, we investigated algebras of partial functions for a signature that we believe provides the necessary order and compatibility structure. The signature has two operations, both binary: the standard set-theoretic \emph{relative complement} operation and a \emph{domain restriction} operation. We gave and proved a finite equational axiomatisation for the class of isomorphs of such algebras of partial functions \cite[\begin{NoHyper}Theorem~5.7\end{NoHyper}]{diff-rest1}.

In the present paper we continue our project with an investigation of `discrete' duality---a term used for dualities requiring no topological information on the duals. A secondary component of our thesis relates specifically to such dualities. Recall the prototypical discrete duality between complete atomic Boolean algebras and sets, and that this extends to a duality between complete atomic Boolean algebras with completely additive operators and Kripke frames. A first observation is that these dualities are just specialisations of (contravariant) adjunctions, where we drop the completeness requirement on the algebra side. A second observation is that for Boolean algebras, the `atomic' condition that remains is an intrinsic/first-order characterisation of the more extrinsic/semantic condition of being \emph{completely representable} (arbitrary cardinality joins/meets become unions/intersections, respectively). Hence we posit the general principle that the correct class to use for a discrete adjunction is always the class of completely representable algebras.

Following this reasoning, in \cite{diff-rest1} we identified the completely representable algebras of our class. Just as for Boolean algebras, they are exactly the atomic ones \cite[\begin{NoHyper}Theorem~6.16\end{NoHyper}]{diff-rest1}. The main results of the present paper are the elaboration of a discrete adjunction between this class of atomic representable algebras and a certain class of set quotients (\Cref{t:adj}) and the extension of that theorem to algebras with additional operators (\Cref{thm:expansion}). We also show that, as for Boolean algebras, the monad induced by the adjunction gives an appropriate form of completion (the \emph{compatible completion}) of algebras (\Cref{p:completion}/\Cref{p:completion'}) and that the adjunction restricts to a duality on the compatibly complete algebras (\Cref{t:discrete-duality}/\Cref{c:extended-duality}).

\subsubsection*{Structure of paper} \Cref{preliminaries} contains preliminaries, including formal definitions of the classes of representable and of completely representable algebras. We recall the axiomatisations of these two classes as presented in \cite{diff-rest1}.

In \Cref{sec:duality}, we present and prove our central result: the adjunction between the atomic representable algebras and a category of set quotients (\Cref{t:adj}).

 \Cref{sec:completion} concerns completion. We define the notions of compatibly complete (\Cref{def:comp}) and of a compatible completion (\Cref{def:completion}), and we prove that compatible completions are unique up to isomorphism (\Cref{p:1_0}). We prove that the monad induced by our adjunction yields the compatible completion on atomic representable algebras (\Cref{p:completion}), and we conclude that the adjunction restricts to a duality on the compatibly complete atomic representable algebras (\Cref{t:discrete-duality}).
 
 \Cref{sec:operators} concerns additional operations. We define the notion of a compatibility preserving completely additive operator (\Cref{def:compatibility-preserving}) and extend the adjunction (\Cref{thm:expansion}), completion (\Cref{p:completion'}), and duality (\Cref{c:extended-duality}) results of the previous two sections to representable algebras equipped with such operators.

\section{Algebras of functions}\label{preliminaries}

Given an algebra $\algebra{A}$, when we write $a \in \algebra{A}$ or say that $a$ is an element of $\algebra{A}$, we mean that $a$ is an element of the domain of $\algebra{A}$. Similarly for the notation $S \subseteq \algebra{A}$ or saying that $S$ is a subset of $\algebra{A}$. We follow the convention that algebras are always nonempty. If $S$ is a subset of the domain of a map $\theta$ then $\theta[S]$ denotes the set $\{\theta(s) \mid s \in S\}$. We use $\join$ and $\meet$ respectively as our default notations for joins (suprema) and meets (infima).

We begin by making precise what is meant by partial functions and
algebras of partial functions. 
\begin{definition}
  Let $X$ and $Y$ be sets. A \defn{partial function} from $X$ to $Y$
  is a subset $f$ of $X \times Y$ validating
\begin{equation*}
  (x, y) \in f \textrel{and} (x, z) \in f \implies y = z.
\end{equation*}
If $X = Y$ then $f$ is called simply a partial function on $X$. 
For a partial function $f \subseteq X \times Y$, if $(x,y)$ belongs to
$f$ then we may write $y = f(x)$. 
Given such a partial function, its \defn{domain} is the set
\[\dom(f) \coloneqq \{x \in X \mid \exists \ y \in Y \from (x, y) \in f\}.\]
For any binary relation $R \subseteq X \times Y$, we write $R^{-1}$
for the relation $\{(y, x) \mid (x, y) \in R\}$. Notice that a partial
function $f \subseteq X \times Y$ is injective if and only if $f^{-1}$
is also a partial function. Finally, for any binary relations $R
\subseteq X \times Y$ and $S \subseteq Y \times Z$, we denote by $S
\circ R$ (or simply $SR$) the composition of $R$ and $S$:
\[S \circ R \coloneqq \{(x, z) \in X \times Z \mid \exists y \in Y \from (x,
  y) \in R \textrel{and}(y, z) \in S\}.\]
When $R$ and $S$ are partial functions, this is their usual composition.
\end{definition}

\begin{definition} 
  An \defn{algebra of partial functions} of the signature $\{-,
  \rest\}$ is a universal algebra $\algebra A = (A, -, \rest)$ where
  the elements of the universe $A$ are partial functions from some
  (common) set $X$ to some (common) set $Y$ and the interpretations of
  the symbols are given as follows:
  \begin{itemize}
    
  \item The binary operation $-$ is \defn{relative complement}:
    \[f - g \coloneqq \{(x, y) \in X \times Y \mid (x, y) \in
      f\textrel{and} (x, y) \not\in g\}.\]

  \item The binary operation $\rest$ is \defn{domain
      restriction}.\footnote{This operation has
        historically been called \emph{restrictive multiplication},
        where \emph{multiplication} is the historical term for
        \emph{composition}. But we do not wish to emphasise this
        operation as a form of composition.} It is the
    restriction of the second argument to the domain of the first;
    that is:
    \[ f \rest g \coloneqq \{(x, y) \in X \times Y \mid x \in \dom(f)
      \textrel{and} (x, y) \in g\}\text{.}\]
  \end{itemize}
\end{definition}
Note that in algebras of partial functions of the signature $\{-,
\rest\}$, the set-theoretic intersection of two elements $f$ and $g$
can be expressed as $f - (f - g)$. We use the symbol~$\cdot$ for this
derived operation.

We also observe that, without loss of generality, we may assume $X =
Y$ (a common stipulation for algebras of partial functions). Indeed,
if $\cA$ is a $\{-, \rest\}$-algebra of partial functions from $X$ to
$Y$, then it is also a $\{-, \rest\}$-algebra of partial functions
from $X \cup Y$ to $X \cup Y$. In this case, this
non-uniquely-determined single set is called `the'
\defn{base}. However, certain properties may not be preserved by
changing the base. For instance, while a partial function is injective
as a function on~$X$ if and only if it is injective as a function
on~$X'$, this is not the case for surjectivity.

\begin{definition}
  An algebra $\algebra A$ of the signature $\{ -, \rest\}$ is
  \defn{representable} (by partial functions) if it is isomorphic to
  an algebra of partial functions. An isomorphism from $\algebra A$ to
  an algebra of partial functions is a \defn{representation} of
  $\algebra A$.
\end{definition}

Just as for algebras of partial functions, for any $\{-, \rest\}$-algebra
$\cA$, we will consider the derived operation $\bmeet$ defined by
\begin{equation*}
  a \bmeet b \coloneqq a - (a - b).
\end{equation*}
In \cite{diff-rest1} it was shown that the
class of $\{-, \rest\}$-algebras that is representable by partial
functions is axiomatised by the following set of
equations.
\begin{enumerate}[label = (Ax.\arabic*), leftmargin = *]
\item \label{schein1} $a - (b - a) = a$
\item \label{commutative} 
$a \bmeet b = b \bmeet a$
\item \label{schein3} $(a - b) - c = (a - c) - b$
\item \label{eq:8} $(a \rest c)\bmeet(b \rest c) = (a \rest b) \rest
  c$
\item \label{lifting}$ (a \bmeet b) \rest a = a \bmeet b$
\end{enumerate}

\begin{theorem}[{\cite[\begin{NoHyper}Theorem~5.7\end{NoHyper}]{diff-rest1}}]
  The class of $\{ -, \rest\}$-algebras representable by partial
  functions is a variety, axiomatised by the finite set of equations
  \ref{schein1} -- \ref{lifting}.
\end{theorem}

Algebras satisfying axioms \ref{schein1} -- \ref{schein3} are called
\defn{subtraction algebras} and it is known that in those algebras the $\bmeet$ operation gives a semilattice structure (which we view as a meet-semilattice), and the
downsets of the form $a^\downarrow \coloneqq\{x \mid x \leq a\}$ are Boolean
algebras~\cite{schein1992difference}. In particular, the same holds
for representable $\{-, \rest\}$-algebras.

It has long been known that in a representable algebra the operation
$\rest$ is associative (see, for example,
\cite{vagner1962}). Moreover, the inequality $a \rest b \leq b$ is
valid (an algebraic proof appears in \cite{diff-rest1}) and will be
often used without further mention.

The next two definitions apply to any function between posets $\cP$
and $\cQ$. So in particular, these definitions apply to homomorphisms
of Boolean algebras and homomorphisms of representable $\{ -,
\rest\}$-algebras.  We denote meets and joins in~$\cP$ by~$\meet$
and~$\join$ respectively, and meets and joins in~$\cQ$ by~$\bigcap$
and~$\bigcup$ respectively.

\begin{definition}\label{def:meet}
  A function $h \from \cP \to \cQ$ is \defn{meet complete} if, for every
  nonempty subset~$S$ of~$\algebra{P}$, if $\meet S$ exists, then so
  does $\bigcap h[S]$ and
  \[h(\meet S) = \bigcap h[S]\text{.}\]
\end{definition}

\begin{definition}\label{def:join}
  A function $h \from \cP \to \cQ$ is \defn{join complete} if, for every
  subset~$S$ of~$\algebra{P}$, if $\join S$ exists, then so does
  $\bigcup h[S]$ and
  \[h(\join S) = \bigcup h[S].\]
\end{definition}

Note that $S$ is required to be nonempty in \Cref{def:meet}, but not
in \Cref{def:join}. For homomorphisms of Boolean algebras, being
meet complete is equivalent to being join complete, and in \cite{diff-rest1} we showed that the same is true for homomorphisms of representable $\{-, \rest\}$-algebras (\begin{NoHyper}Corollary~6.5\end{NoHyper} and \begin{NoHyper}Corollary~6.6\end{NoHyper} there). So in these cases we may simply describe such a
homomorphism using the adjective \defn{complete}.\footnote{In the case
  of representations of Boolean algebras as fields of sets, the
  adjectives  \emph{strong} and  \emph{regular} have also been used.}
  
  \begin{definition}
A \defn{complete representation} of an algebra of the signature $\{-, \rest\}$ is a representation $\theta$, with base $X$ say, such that $\theta$ forms a complete homomorphism when viewed as an embedding into the algebra of \emph{all} partial functions on $X$. An algebra is \defn{completely representable} if it has a complete representation. 
\end{definition}

Complete representations have been studied previously in the context of various different forms of representability: by sets \cite{egrot}, by binary or higher-order relations \cite{MR1330986, journals/jsyml/HirschH97a}, or by partial functions \cite{complete}.
   
   \begin{definition}
Let $\algebra{P}$ be a poset with a least element, $0$. An \defn{atom} of $\algebra{P}$ is a minimal nonzero element of $\algebra{P}$. We write $\At(\algebra{P})$ for the set of atoms of $\algebra P$. We say that $\algebra{P}$ is \defn{atomic} if every nonzero element is greater than or equal to an atom. 
\end{definition}

   In \cite{diff-rest1}, it was shown that a $\{-, \rest\}$-algebra is completely representable if and only if it is both representable and atomic.
   
   \begin{theorem}[{\cite[\begin{NoHyper}Theorem~6.16\end{NoHyper}]{diff-rest1}}]
   The class of $\{ -, \rest\}$-algebras that are completely representable by partial functions is axiomatised by the finite set of equations \ref{schein1} -- \ref{lifting} together with the $\forall\exists\forall$ first-order formula stating that the algebra is atomic.
\end{theorem}

This theorem justifies our interest in the atomic representable
algebras, and also indicates that complete homomorphism is the
appropriate notion of morphism between these algebras.

This is a good place to define two further terms that we will use later.

   \begin{definition}\label{def:atomistic}
Let $\algebra{P}$ be a poset. A subset $S$ of $\algebra P$ is \defn{join dense} (in $\algebra P$) if each $p \in \algebra P$ is the join $\join T$ of some subset $T$ of $S$. The poset $\algebra P$ is \defn{atomistic} if $\At(\algebra{P})$ is join dense in $\algebra P$.
\end{definition}

Of course, atomic implies atomistic, for any poset. For representable $\{-, \rest\}$-algebras, the converse is also true \cite[\begin{NoHyper}Lemma~6.13\end{NoHyper}]{diff-rest1}, generalising the same statement for Boolean algebras.

\section{Discrete adjunction for atomic representable
  algebras}\label{sec:duality}
In this section we exhibit a contravariant adjunction between the atomic
representable algebras, $\aralg$, and a certain category $\setq$ whose
objects are quotients of sets.

We use $\twoheadrightarrow$ to indicate a (total) surjective function between sets, and we use $\hookrightarrow$ to indicate an embedding of algebras. The notation $\parrow$ indicates a partial function. 
 We may at times use a bracket-free notation for applications of functors to morphisms, for example, $Fh$ in place of $F(h)$.

Before giving the claimed adjunction, we define the two categories involved.

\begin{definition}
  We denote by $\aralg$ the category whose objects are atomic
   $\{-, \rest\}$-algebras representable by partial functions, and whose morphisms are
  complete homomorphisms of $\{-, \rest\}$-algebras.
\end{definition}

\begin{definition}\label{category2}
  We denote by $\setq$ the category whose objects are set quotients (that is, surjective functions between sets)
  $\pi \from X \twoheadrightarrow X_0$, and where a morphism from $\pi \from  X
  \twoheadrightarrow X_0$ to $\rho \from Y \twoheadrightarrow Y_0$ is a
  partial function $\varphi \from X \parrow Y $ satisfying the following
  conditions:
  \begin{enumerate}[label = (Q.\arabic*)]
  \item\label{item:Q1} 
  $\varphi$ \defn{preserves equivalence}: if both $\varphi(x)$ and $\varphi(x')$ are defined, then
  \[ \pi(x) = \pi(x') \implies \rho(\varphi(x)) = \rho(\varphi(x')).\]
  In particular, $\varphi$ induces a partial function
  $\widetilde\varphi \from X_0 \parrow Y_0$ given by
  \[\widetilde\varphi \coloneqq \{(\pi(x), \rho(\varphi(x))) \mid x \in
    \dom (\varphi)\}.\]
\item\label{item:Q2} $\varphi$ is \defn{fibrewise injective}: for
  every $(x_0, y_0) \in \widetilde \varphi$, the restriction and
  co-restriction of $\varphi$ induces an injective partial map
  \[\varphi_{(x_0, y_0)} \from  \pi^{-1}(x_0)\parrow \rho^{-1}(y_0),\]
\item\label{item:Q3} $\varphi$ is \defn{fibrewise surjective}: for
  every $(x_0, y_0) \in \widetilde \varphi$, the induced partial map
  $\varphi_{(x_0, y_0)}$ is surjective (that is, the image of $\varphi_{(x_0, y_0)}$ is the whole of $\rho^{-1}(y_0)$).\footnote{Note
      that in the context of partial maps, the conjunction of
      `injective' and `surjective' is not `bijective' in the sense of
      a one-to-one correspondence.}
  \end{enumerate}
\end{definition}


In what follows, we define two functors $F \from \aralg \to
\setq^{\operatorname{op}}$ and $G \from \setq^{\operatorname{op}} \to
\aralg$, which we then show to form an adjunction
(\Cref{t:adj}).

Before defining $F$, we first recall some notation from \cite{diff-rest1}. 

\begin{definition}\label{sec:equiv}
  Given a representable $\{-, \rest\}$-algebra $\algebra A$, the
  relation $\preceq_\cA$ on $\algebra A$ is defined by
  \[a \preceq_\cA b \iff a \leq b \rest a\]
  and is a preorder.  We denote by $\sim_\cA$ the equivalence
  relation induced by $\preceq_\cA$, and for a given $a \in \cA$ we
  use $[a]$ to denote the equivalence class of~$a$.
\end{definition}

In the case that $\algebra A$ is an actual $\{-, \rest\}$-algebra of partial functions, the relation $\preceq_\cA$ is the \emph{domain inclusion} relation $f \preceq_\cA g \iff \dom(f) \subseteq \dom(g)$.\footnote{See \cite{Schein1970} for axiomatisations for various signatures containing domain inclusion (as a fundamental relation).}

 The next result summarises some facts about $\preceq_\cA$ that were
proved in~\cite{diff-rest1} and will be used later.

\begin{proposition}\label{p:1}
  The following statements hold for a representable $\{-,
  \rest\}$-algebra~$\cA$.
  \begin{enumerate}[label = (\alph*)]
  \item\label{item:2} The relation $\leq$ is contained in
    $\preceq_\cA$, and $[0] = \{0\}$.
  \item\label{item:3} The poset $\cA/{\sim_\cA}$ is a meet-semilattice
    (actually a subtraction algebra) with meet given by $[a]\wedge [b]
    = [a \rest b]$.
  \item\label{item:4} The relations $\leq$ and $\preceq_\cA$ coincide
    in each downset $a^\downarrow$. Moreover, the assignment $b
    \mapsto [b]$ provides an isomorphism between the Boolean algebras
    $a^\downarrow$ and $[a]^\downarrow$.
  \end{enumerate}
\end{proposition}

\subsection{The functor \texorpdfstring{$F \from \aralg \to
    \setq^{\operatorname{op}}$}{F}}\label{sec:F}\hfill\par
\smallskip We let $F \from \aralg \to \setq^{\operatorname{op}}$ be
defined as follows. Given an atomic representable algebra $\cA$, the
set quotient $F(\cA)$ is the canonical projection $\pi_\cA \from
\At(\cA) \twoheadrightarrow \At(\cA)/{\sim_\cA}$. This defines $F$ on
the objects.
By~\Cref{p:1}\ref{item:4}, for a Boolean algebra $\cB$, the relation
$\sim_\cB$ is the identity and thus the restriction of $F$ to atomic
Boolean algebras is simply $\At(\_)$.
For defining $F$ on the morphisms, we let $h \from \cA \to \cB$ be a
complete homomorphism of atomic representable algebras. By
\cite[\begin{NoHyper}Lemma~6.3\end{NoHyper}]{diff-rest1}, for each $a \in \cA$, the restriction of $h$ induces a
complete homomorphism of (atomic) Boolean algebras $h_a \from a^{\downarrow}
\to h(a)^{\downarrow}$. We denote by $\varphi_a:
\At(h(a)^{\downarrow}) \to \At(a^{\downarrow})$ its discrete
dual. Recall that $\varphi_a$ is completely determined by the
following Galois connection.
\begin{equation}
  \forall \ a' \in a^{\downarrow},\ y \in \At(h(a)^{\downarrow})\quad
  (\varphi_a(y) \le a' \iff y \le h(a'))\label{eq:disc-dual}
\end{equation}
In particular, for every $a \in \cA$, the map $h \circ \varphi_a$ is a closure
operator.

Then $Fh \from \At(\cB) \parrow \At(\cA)$ has domain $\cB_0 = \bigcup_{a
  \in \cA}\At(h(a)^{\downarrow})$, and for $y \in
\At(h(a)^{\downarrow})$, we set $Fh(y) =
\varphi_a(y)$.
We observe that this is a well-defined map, that is, if $a_1, a_2 \in
\cA$ are such that $y \in \At(\cB)$ is below $h(a_1)$ and $h(a_2)$,
then $\varphi_{a_1}(y) = \varphi_{a_2}(y)$. Indeed, since $h$ is a
homomorphism, we have $y \leq h(a_1\cdot a_2)$, and
by~\eqref{eq:disc-dual} this yields $\varphi_{a_2}(y) \leq a_1\cdot
a_2$. Using again~\eqref{eq:disc-dual} twice, we have
\[\varphi_{a_2}(y) \leq \varphi_{a_2}(y) \iff y \leq
  h(\varphi_{a_2}(y)) \iff \varphi_{a_1}(y) \leq \varphi_{a_2}(y).\]
Similarly, we can prove the inequality $\varphi_{a_2}(y) \leq
\varphi_{a_1}(y)$, thereby concluding that $Fh$ is well-defined.
As a consequence of $h \circ \varphi_a$ being a closure operator for
each $a$, we have the following.
\begin{equation}
  \label{eq:7}
  \forall\ y \in \cB_0\quad y \leq (h \circ Fh)(y)
\end{equation}

We now prove that $Fh$ defines a morphism in $\setq$ from $\pi_\cB$ to
$\pi_\cA$.  Let $y_1, y_2 \in \cB_0$ be $\sim_\cB$-equivalent, that
is, $y_1 = y_2 \rest y_1$ and $y_2 = y_1 \rest y_2$. Since $y_i \leq
(h\circ Fh)(y_i)$, for $i = 1,2$, using the fact that $\rest$ is order
preserving in both coordinates and the fact that $h$ is a
homomorphism, we have
\[y_1 = y_2 \rest y_1 \leq h(Fh(y_2)) \rest h(Fh(y_1))= h(Fh(y_2)
  \rest Fh(y_1)).\] We may apply~\eqref{eq:disc-dual} to obtain
\[Fh(y_1) \leq Fh(y_2) \rest Fh(y_1),\]
that is, $Fh(y_1)\preceq_\cA Fh(y_2)$. Likewise, we can show that
$Fh(y_2)\preceq_\cA Fh(y_1)$, and thus $Fh(y_1) \sim_\cA Fh(y_2)$. This proves~\ref{item:Q1}.

Finally, to conclude that $Fh$ defines a morphism in $\setq$, we only
need to show that for every $(y_0, x_0) \in \widetilde{Fh}$, the partial map
$Fh_{(y_0, x_0)} \from \pi_\cB^{-1}(y_0) \parrow \pi_\cA^{-1}(x_0)$ is
both injective and surjective. (Recall that the subscript $(y_0, x_0)$ indicates restricting and co-restricting to the fibres of $y_0$ and $x_0$ respectively; see \Cref{category2}\ref{item:Q2}.) Injectivity of $Fh_{(y_0, x_0)}$ is a
consequence of \Cref{p:1}\ref{item:4} together with the observation
that if $y_1, y_2 \in \cB_0$ are such that $x\coloneqq Fh(y_1) = Fh(y_2)$,
then $y_1$ and $y_2$ have the common upper bound $h(x)$. 

To see that $Fh_{(y_0, x_0)}$ is surjective, take $x \in
\pi_\cA^{-1}(x_0)$.  As $\widetilde{Fh}$ is defined on $y_0$, by the
definition of $\widetilde{Fh}$ from $Fh$ we know $Fh$ is defined on at
least one element $y$ of $\pi_\cB^{-1}(y_0)$, and that $Fh(y)
\sim_{\algebra A} x$. We show that $Fh_{(y_0, x_0)}(y \rest h(x)) =
x$. Using~\eqref{eq:7} and then applying the fact that $h$ is a
homomorphism to $Fh(y) \sim_{\algebra A} x$, we obtain $y \leq (h\circ
Fh)(y) \sim_{\algebra B} h (x)$, from which we obtain $y \preceq_\cB
h(x)$. By
\cite[\begin{NoHyper}Lemma~6.14(b)\end{NoHyper}]{diff-rest1}, $y \rest
h(x)$ is an atom of $\algebra B$ (hence of $h(x)^\downarrow$), which
is $\sim_\cB$-equivalent to $y$.  From $y \rest h(x) \leq h(x)$ we
obtain $\varphi_{x}(y \rest h(x)) \leq x$ and hence $Fh(y \rest h(x))
= x$. Since $y \rest h(x) \sim_{\algebra B} y$ we have $Fh_{(y_0,
  x_0)}(y \rest h(x)) = x$, as required.

To summarise, we have proved the following.

\begin{proposition}
  There is a functor \[F \from \aralg \to \setq^{\operatorname{op}}\]
  that maps an atomic representable $\{-,\rest\}$-algebra $\cA$ to
  the canonical projection $\pi_\cA \from \At(\cA) \twoheadrightarrow
  \At(\cA)/{\sim_\cA}$.
\end{proposition}

\subsection{The functor \texorpdfstring{$G \from \setq^{\operatorname{op}} \to
\aralg$}{G}}\hfill\par
\smallskip We now define the functor $G \from
\setq^{\operatorname{op}} \to \aralg$. For a set quotient $\pi \from X
\twoheadrightarrow X_0$, we let $G(\pi) = \cA_\pi$ be the algebra of
partial functions consisting of all partial functions $f \from X_0
\parrow X $ that are a subset of $\pi^{-1} = \{(\pi (x), x) \mid x \in X
\}$.  Since for every $f, g \in \cA_\pi$, both $f-g$ and $g \rest f$
are subsets of $f$, and $\cA_\pi$ is closed under subsets, $f-g$ and
$g \rest f$ belong to $\cA_\pi$.
\begin{lemma}
  For every quotient $\pi \from X \twoheadrightarrow X_0$, the $\{-,
  \rest\}$-algebra $\cA_\pi$ is atomic and representable.
\end{lemma}
\begin{proof}
  The algebra $\cA_\pi$ is an actual algebra of partial functions,
  hence $\cA_\pi$ is, trivially, representable.  The minimal element
  of $\cA_\pi$ is of course the empty function. The minimal nonzero
  elements are then precisely the singletons $\{(\pi (x),x)\}$, for $x
  \in X$. Every nonzero element of $\cA_\pi$ includes one of these
  singletons, and hence $\cA_\pi$ is atomic.
\end{proof}

Let us define $G$ on morphisms. Given a morphism $\varphi$ from $(\pi \from
X \twoheadrightarrow X_0)$ to $(\rho \from Y \twoheadrightarrow Y_0)$ in
$\setq$ (that is, a morphism from $\rho$ to $\pi$ in
$\setq^{\operatorname{op}}$), we let $G\varphi \from \cA_{\rho} \to
\cA_\pi$ assign to each partial function $g \in \cA_\rho$, the partial
function $G\varphi (g) \from X_0\parrow X$
given by
\begin{equation*}
  G\varphi(g) = \{(\pi (x), x) \in X_0 \times X \mid \exists y \in Y
  \from (x,y) \in \varphi \textrel{and} (\rho (y), y) \in g
  \}.
\end{equation*}
Using the fact that $\varphi$ is injective on each fibre of $\pi$, it
is easy to check that $G\varphi (g)$ is indeed a partial function and
thus an element of $\cA_\pi$. We remark that since $\varphi$ is a
partial function, when $(\pi (x), x)$ belongs to $G\varphi (g)$ there
is in fact \emph{a unique} $y \in Y$, namely $y = \varphi(x)$, so that
$(x,y) \in \varphi$ and $(\rho (y), y) \in g$. Thus, $G\varphi$ may be
alternatively described as
\begin{equation}
 G\varphi (g) = \{(\pi (x), x) \mid x \in \dom (\varphi) \textrel{and}
  (\rho  (\varphi( x)), \varphi (x)) \in g\}.\label{eq:17}
\end{equation}
\begin{lemma}\label{l:1}
  For every morphism $\varphi$ as above, the function $G\varphi$
  defines a complete homomorphism of $\{-, \rest\}$-algebras.
\end{lemma}
\begin{proof}
  The fact that $G\varphi$ preserves $-$ is a trivial consequence
  of~\eqref{eq:17}. To see that $G\varphi$ preserves $\rest$, it
  suffices to show that, for every $x \in \dom(\varphi)$ and $g \in
  \cA_\rho$, we have
  \begin{equation}\label{eq:equivalence}\rho (\varphi (x)) \in \dom (g) \iff \pi (x) \in
    \dom(G\varphi (g)).\end{equation}
  Indeed, given any $f, g \in \algebra A_\rho$, the two
  partial-function images $G\varphi(g \rest f)[X_0]$ and $(G\varphi(g)
  \rest G\varphi(f))[X_0]$ are subsets of $\varphi^{-1}(f[Y_0])$, and
  for all $x \in \varphi^{-1}(f[Y_0])$ (that is, for all $x \in \dom
  (\varphi)$ such that $(\rho(\varphi(x)), \varphi(x)) \in f$) we have
  $x \in G\varphi(g \rest f)[X_0] \iff \rho(\varphi(x)) \in \dom(g)$
  and $x \in (G\varphi(g) \rest G\varphi(f))[X_0] \iff \pi(x) \in
  \dom(G\varphi (g))$. Thus the images of the injective partial functions
  $G\varphi(g \rest f)$ and $G\varphi(g) \rest G\varphi(f)$ are equal
  provided~\eqref{eq:equivalence} holds.

  To prove \eqref{eq:equivalence}, first suppose that $\rho (\varphi (x)) \in \dom(g)$. Then there exists
  $y \in Y$ so that $(\rho (\varphi (x)), y) \in g \subseteq \rho^{-1}$,
  and since $\varphi$ is fibrewise surjective by~\ref{item:Q3}, there
  exists $x' \in \dom (\varphi)$ so that $\pi (x') = \pi (x)$ and $\varphi
  (x') = y$.  Then $(\pi (x'), x') \in G\varphi(g)$, so $\pi (x) = \pi (x')$
  belongs to the domain of $\dom(G\varphi (g))$. Conversely, if $\pi (x)
  \in \dom (G\varphi (g))$, then there exists some $x' \in \dom
  \varphi$ such that $\pi (x) = \pi (x')$ and $(\rho (\varphi (x')), \varphi
  (x')) \in g$. Using~\ref{item:Q1}, we have $\rho (\varphi (x')) = \rho
 ( \varphi (x))$, implying that $\rho( \varphi (x)) \in\dom (g)$ as
  required.

  Finally, since infima in $\cA_\pi$ and $\cA_\rho$ are given by
  intersections, by the definition of $G\varphi$ we easily conclude that
  $G\varphi$ is a meet-complete, hence complete, homomorphism.
\end{proof}

We have proved the following.

\begin{proposition}
  There is a functor
  \[G \from \setq^{\operatorname{op}} \to \aralg\]
  that sends a quotient $\pi \from X \twoheadrightarrow X_0$ to the $\{-,
  \rest\}$-algebra $\cA_\pi$ consisting of all partial functions
  included in $\pi^{-1} = \{(\pi (x), x) \mid x \in X\}$.
\end{proposition}

We can now prove that we have the claimed adjunction.
\begin{theorem}\label{t:adj}
  The functors $F \from \aralg \to \setq^{\operatorname{op}}$ and $G
  \from \setq^{\operatorname{op}} \to \aralg$ form an adjunction. That
  is
  \[F \dashv G.\]
\end{theorem}

\begin{proof}
  We only need to define natural transformations $\eta \from {\rm
    Id}_{\aralg} \implies G \circ F$ and $\lambda \from {\rm Id}_{\setq} \implies
  F \circ G$ satisfying
  \begin{equation}
    F \eta \circ \lambda_F = {\rm Id}_F \qquad \text{and} \qquad G\lambda
    \circ \eta_G = {\rm Id}_G.\label{eq:1}
  \end{equation}

  Let us first define $\eta$.  Notice that, if $\cA$ is an atomic and
  representable algebra, the algebra $(G \circ F)(\cA)$ consists of all
  partial functions contained in $\{([x], x) \mid x \in
  \At(\cA)\}$. Thus we take for $\eta_\cA$ the representation map of
  \cite[\begin{NoHyper}Corollary~6.17\end{NoHyper}]{diff-rest1}, that is,
  \[\eta_\cA(a) \coloneqq  \{([x], x) \mid x \in \At(\cA)
    \textrel{and} x \leq a\}.\]

  In order to define $\lambda$, we observe that for a set quotient
  $\pi \from X \twoheadrightarrow X_0$, since the atoms of $G(\pi) = \cA_\pi$
  are precisely the singletons $a_x \coloneqq \{(\pi (x), x)\}$, for $x \in X$,
  we have $(a_{x_1} \sim_{\cA_\pi} a_{x_2} \iff \pi (x_1) = \pi (x_2))$,
  and thus $(F \circ G)(\pi)$ is a quotient isomorphic to $\pi$.  Therefore
  $\lambda_\pi \from X \parrow (F \circ G)(\pi)$ is simply the identity (total) function
  on $X$ (or, strictly speaking, the function that maps $x \in X$ to
  $a_x \in (F \circ G)(\pi)$).
  
  Finally, checking that $\eta$ and $\lambda$ as defined are indeed
  natural transformations satisfying~\eqref{eq:1} amounts to a tedious
  but routine computation.
\end{proof}

\section{Compatible completion and discrete duality}\label{sec:completion}
 
 In this section we first define the appropriate analogues of completeness and completion for algebras with compatibility information. Then we show that, similarly to the contravariant adjunction between atomic Boolean algebras and
sets, the monad induced by $F \from \aralg \dashv \setq^{\operatorname{op}} \,\from\! G$ gives precisely the `completion' of the
algebra. In fact, our contravariant
adjunction is a generalisation of the one between atomic Boolean
algebras and sets. Just like that adjunction, ours restricts to a duality of the full subcategory of (compatibly) complete algebras.

In order to define analogues of completeness and completion in the presence of compatibility information, we first need to formalise this idea of compatibility.

\begin{definition}
  Let $\cP$ be a poset. A binary relation $C$ on $\cP$ is a
  \defn{compatibility relation} if it is reflexive, symmetric, and
  downward closed in $\cP \times \cP$.  We say that two elements $a_1,
  a_2 \in \cP$ are \defn{compatible} if $a_1 C a_2$.
  \end{definition}

One can show that `reflexive,
symmetric, and downward closed' is an axiomatisation of
a rather general conception of compatibility in the following sense.

\begin{proposition} Let $(P, \leq, C)$ be a poset equipped with a
  binary relation $C$. Then $(P, \leq, C)$ is isomorphic to a poset
  $(P', \subseteq, C')$ of partial functions ordered by inclusion and
  equipped with the relation `agree on the intersection of their
  domains' if and only if $C$ is reflexive, symmetric, and downward
  closed.\footnote{The relation $C'$ was introduced
      in~\cite{Vagner1960} for semigroups of partial transformations,
      where it was called
      \emph{semi-compatibility}. 
  }
\end{proposition} 

\begin{proof}
  It is clearly necessary that $C$ be reflexive, symmetric, and
  downward closed. Showing these conditions are sufficient can be
  proved with the map sending each $p \in P$ to the partial function
\[\theta(p) \coloneqq \{(\{p'\}, p') \mid p' \leq p\} \cup \{(\{p', q\},p') \mid p' \leq p \text{ and } p' \mathbin{\slashed C} q\}.\]
These are indeed partial functions: a pair of type $(\{p', q\},p')$
cannot also be of type $(\{p'\}, p')$, for $p' = q$ contradicts $p'
\mathbin{\slashed C} q$; so we only need check that $\theta(p)$ never
contains both the pairs $(\{p', q\},p')$ and $(\{p', q\},q)$, for $p'
\neq q$. But containing both would imply $p' \leq p$ and $q \leq p$,
which by reflexivity and downward closure of~$C$ contradicts $p'
\mathbin{\slashed C} q$.

That $\theta$ is injective is ensured by pairs of the first type, for if $p_1 \neq p_2$ then without loss of generality $p_1 \not\leq p_2$, so that $(\{p_2\}, p_2)$ is in $\theta(p_2)$ but not $\theta(p_1)$. It is clear that $p_1 \leq p_2$ implies $\theta(p_1) \subseteq \theta(p_2)$. 

Lastly, we show that $p_1 \mathbin{\slashed C} p_2$ if and only if $\theta(p_1) \mathbin{\slashed C'} \theta(p_2)$. Suppose $p_1 \mathbin{\slashed C} p_2$. Then in particular $p_1 \neq p_2$. But by the supposition (and symmetry of the $C$ relation) $(\{p_1, p_2\}, p_1) \in \theta(p_1)$ and $(\{p_1, p_2\}, p_2) \in \theta(p_2)$. So $\theta(p_1)$ and $\theta(p_2)$ disagree at the point $\{p_1, p_2\}$. Conversely, suppose $\theta(p_1) \mathbin{\slashed C'} \theta(p_2)$. Then the disagreement can only be at points of cardinality two. That is, for some $p'_1, p'_2$ we have $(\{p'_1, p'_2\}, p'_1) \in \theta(p_1)$ and $(\{p'_1, p'_2\}, p'_2) \in \theta(p_2)$. But that implies $p_1' \mathbin{\slashed C} p_2'$, which by downward-closure of compatibility gives $p_1 \mathbin{\slashed C} p_2$.
\end{proof}

 \begin{definition}\label{def:comp} A poset $\cP$ equipped with a compatibility relation is said
  to be \defn{compatibly complete} provided it has joins of all
  subsets of pairwise-compatible elements. We say $\cP$ is \defn{meet
    complete} if it has meets of all \emph{nonempty} subsets.
\end{definition}

\begin{definition}\label{algebra_compatibility}When speaking about compatibility for representable
  $\{-, \rest\}$-algebras, we mean the relation that makes two
  elements compatible precisely when
  \[ a_1 \rest a_2 = a_2 \rest a_1.\]
\end{definition}

It is clear that for an actual $\{-, \rest\}$-algebra of partial functions, two elements are compatible
exactly when they agree on their shared domain. This is the easiest way to see that \Cref{algebra_compatibility} does indeed define a compatibility relation.

Note that in the case that all pairs of elements are compatible,
\emph{compatibly complete} is equivalent to \emph{complete}. Boolean
algebras provide examples of representable $\{-, \rest\}$-algebras
where every pair is compatible, if $-$ is the Boolean complement and
we use $\rest$ as the meet symbol. In the same way, generalised
Boolean algebras provide a more general class of
examples.\footnote{A \emph{generalised Boolean algebra}
    is a `Boolean algebra without a top', that is, a distributive
    lattice with a bottom and a relative complement operation, $-$,
    validating $a \wedge (b - a) = \bot$ and $a \vee (b - a) = a \vee
    b$.} Thus `compatibly complete' is a coherent generalisation of
  `complete' from situations where there is no compatibility
  information, to those where there is.

In a poset equipped with a compatibility relation, if a subset $S$ has
an upper bound $u$, then by reflexivity and downward closure of
compatibility, $S$ is pairwise compatible. Thus \emph{compatibly
  complete} $\implies$ \emph{bounded complete}. For a similar reason,
\emph{compatibly complete} $\implies$ \emph{directed complete}. So
`compatibly complete' subsumes these two preexisting order-theoretic
concepts.

From now on, we focus exclusively on our specific case of interest: representable  $\{-, \rest\}$-algebras.

\begin{lemma}\label{l:2}
  Let $\algebra A$ be a representable $\{-, \rest\}$-algebra. If
  $\algebra A$ is compatibly complete, then it is meet complete. The
  converse is false.
\end{lemma}

\begin{proof}
  Suppose $\algebra A$ is compatibly complete, and take a nonempty
  subset $S$ of $\algebra A$. Let $S^l$ denote the set of lower bounds
  for $S$. Since $S$ is nonempty, we have an $s \in S$. Then $s$ is an
  upper bound for $S^l$, and so $S^l$ is pairwise compatible. Then by
  the supposition, $S^l$ has a join. It is then straightforward to
  show that this join is the meet of $S$.

  The three-element $\{-, \rest\}$-algebra consisting of the partial functions
  $\emptyset$, $\{(1,1)\}$, and $\{(2,2)\}$ provides a counterexample
  for the converse statement.
\end{proof}

\begin{lemma}\label{l:3}
  Let $\cA$ be a representable $\{-, \rest\}$-algebra. Then $\cA$ is
  meet complete if and only if for every $a \in \cA$ the Boolean algebra
  $a^\downarrow$ is complete.
\end{lemma}
\begin{proof}
  The forward direction holds simply because for every nonempty subset
  $S \subseteq a^\downarrow$ we have $\meet_{a^\downarrow} S
  =\meet_\cA S$, and for the empty subset we have
  $\meet_{a^\downarrow} \emptyset = a$. For the converse, given a
  nonempty subset $S \subseteq \cA$, fix $a \in S$. Then
  $\meet_\cA S = \meet_{a^\downarrow} \{a\cdot s \mid s \in S\}$.
\end{proof}

Next we define what a compatible completion of a representable $\{-,
\rest\}$-algebra is. We are guided by completions of Boolean
algebras. Just as for Boolean algebras,
completions will be unique up to isomorphism (\Cref{p:1_0}).

\begin{definition}\label{def:completion}
  A \defn{compatible completion} of a representable $\{-,
  \rest\}$-algebra $\cA$ is an embedding $\iota \from \algebra A
  \hookrightarrow \algebra C$ of $\{-, \rest\}$-algebras such that $\cC$ is representable and compatibly complete and
  $\iota[\cA]$ is join dense in~$\cC$.
\end{definition}

The next lemma tells us in particular that compatible completions are complete homomorphisms.

\begin{lemma}\label{lem:dense_complete}
Let $\iota \from \algebra A \hookrightarrow \algebra B$ be an embedding of representable $\{-, \rest\}$-algebras. If $\iota[\algebra A]$ is join dense in $\algebra B$ then $\iota$ is complete.
\end{lemma}

\begin{proof}
  Suppose the join of $S \subseteq \algebra A$ exists. Then the
  restriction $\iota \from (\join S)^\downarrow \to \iota(\join
  S)^\downarrow$ is an embedding of Boolean algebras with join dense
  image, so just apply the corresponding known result for Boolean
  algebras, yielding $\iota(\join S) = \join_{\iota(\join
    S)^\downarrow}\iota[S] = \join_{\algebra B}\iota[S]$.
\end{proof}

The following technical results will be used, in particular, for the
proof that compatible completions are unique up to unique
isomorphism. Given an $n$-ary operation $\Omega$ on $\cA$ and subsets $S_1,
\dots, S_n \subseteq \cA$, we shall use $\Omega(S_1, \dots, S_n)$ to
denote the set $\{\Omega(s_1, \dots, s_n) \mid s_1 \in S_1,\dots, s_n
\in S_n\}$, and if one of the sets is a singleton $\{a\}$, we may
simply write~$a$.

\begin{lemma}\label{l:distributive}
  Let $\cA$ be a compatibly complete and representable $\{-, \rest\}$-algebra and $S, T
  \subseteq \cA$ two subsets of pairwise-compatible elements. \begin{enumerate}[(a)]
  \item\label{item:aa}
  The set $S \rest T$ consists of pairwise-compatible elements, and
    \[\join S \rest \join T = \join (S \rest T).\]

  \item\label{item:bb} Suppose $S$ is nonempty. Then for every $a \in
    \cA$ we have
    \[\meet(a-S) = a-\join S.\]
  \end{enumerate}
\end{lemma}

\begin{proof}
  For \ref{item:aa}: the fact that $S \rest T$ is a set whose elements
  are pairwise compatible follows from a repeated use of the equality
  $a \rest b \rest c = b \rest a \rest c$, which is a consequence
  of~\ref{eq:8}. 
  Fix an arbitrary element $a \in \algebra A$ be arbitrary and, for
  the sake of readability, define $s_0 \coloneqq \join S$ and $t_0
  \coloneqq \join T$.  We prove the asserted equation in two steps:
  first we show that $a \rest t_0 = \join (a \rest T)$ and then that
  $s_0 \rest a = \join(S \rest a)$. From there we have
    \begin{align*}
      \join S \rest \join T
      & =  s_0 \rest t_0 =  \join (s_0 \rest T) =
        \join_{t \in T} (s_0 \rest t) \\&= \join_{t \in T} \join (S \rest t)
        = \join_{t \in T} \join_{s \in S}s \rest t = \join S \rest T
    \end{align*}
    as required. (The last equality is a true property of suprema for any doubly indexed set, when all the suprema exist.)
  
  For $a \rest t_0 = \join (a
  \rest T)$, 
  \Cref{p:1}\ref{item:4} states that for any $c \in \algebra A$, the
  map $b \mapsto [b]$ provides an order isomorphism $(c^\downarrow,
  \leq)$ to $([c]^\downarrow, \preceq)$. So given that $t_0$ is an
  upper bound for both sides of $a \rest t_0 = \join (a \rest T)$, and
  $a$ is an upper bound for both sides of $s_0 \rest a = \join(S \rest
  a)$, it is actually sufficient to show that $[a \rest t_0] = [\join
  (a \rest T)]$ and that $[s_0 \rest a] = [\join(S \rest a)]$.
  
  By Lemmas~\ref{l:2} and \ref{l:3}, the poset $t_0^\downarrow$ is a
  complete Boolean algebra, and so $[t_0]^\downarrow$ is as well. The
  meets/joins in $[t_0]^\downarrow$ are written in $\wedge$/$\vee$
  notation; in $[t_0]^\downarrow$ we have that $[b] \wedge [c] = [b
  \rest c]$ is valid (\Cref{p:1}\ref{item:3}). Being careful that all
  terms rest in $[t_0]^\downarrow$, we compute
  \[[a \rest t_0] = [a \rest t_0] \wedge [\join T] = [a \rest t_0]
    \wedge \bigvee_{t \in T}[t] = \bigvee_{t \in T}([a \rest t_0]
    \wedge [t]) = \bigvee_{t \in T}[a \rest t] = [\join (a \rest
    T)],\]
  with the central distributive equality being a valid property of any
  Boolean algebra (provided the join on the left side exists)
  \cite{koppelberg1989handbook}.
 
  Proving the equality $s_0 \rest a = \join(S \rest a)$ is very
  similar. In $[a]^\downarrow$ we have
  \[[\join(S \rest a)] = \bigvee_{s \in S}[s \rest a],\]
  and then in $[s_0]^\downarrow$ we compute
  \[\bigvee_{s \in S}[s \rest a] = \bigvee_{s \in S}([s] \wedge [a
    \rest s_0]) = (\bigvee_{s \in S}[s]) \wedge [a \rest s_0] =[s_0]
    \wedge [a \rest s_0] = [s_0 \rest a].\]
    
    For \ref{item:bb}:   we first observe that in the complete Boolean
  algebra $a^\downarrow$, we have
  \[\meet (a-S) =\meet (a-a\cdot S)=
    a-\join (a \cdot S),\]
  where $a\cdot S$ denotes the set $\{a \cdot s \mid s \in S\}$, and
  the first equality holds because $a - (a\bmeet b) = a-b$ is a valid
  equation in any representable algebra. In turn, in the complete
  Boolean algebra $s_0^\downarrow$, we may compute
  \[\join (a \bmeet S) = \join (a\bmeet s_0 \bmeet S) = a\bmeet s_0
    \bmeet \join S = a \bmeet \join S.\]
  Thus again by the validity of $a - (a\bmeet b) = a-b$ the desired
  equality follows.
\end{proof}

\begin{lemma}\label{lemma:key}
  Let $\iota \from \algebra A \hookrightarrow \algebra B$ and $\iota'
  \from \algebra A \hookrightarrow \algebra C$ be complete embeddings
  of representable $\{-, \rest \}$-algebras, and suppose that
  $\iota[\algebra A]$ is join dense in $\algebra B$, and that
  $\algebra C$ is compatibly complete. Then $\theta \from b \mapsto
  \join \{\iota'(a) \mid \iota(a) \leq b\}$ is a well-defined complete
  embedding of $\{-, \rest\}$-algebras from $\algebra B$ to $\algebra
  C$, and $\iota' = \theta \circ \iota$.
\end{lemma}

\begin{proof}
  First we check $\theta$ is well-defined. The set $\{\iota(a) \mid
  \iota(a) \leq b\}$ is pairwise compatible, as it is bounded
  above. Hence, by injectivity of $\iota$, the set $\{a \mid \iota(a)
  \leq b\}$ is pairwise compatible, and so $\{\iota'(a) \mid \iota(a)
  \leq b\}$ is too. By compatible completeness of $\algebra C$, the
  value $\join \{\iota'(a) \mid \iota(a) \leq b\}$ exists.  It is
  clear that $\theta$ is order preserving and that $\theta \circ \iota
  = \iota'$.
   
  Next we show that $\theta$ is join complete, as a map between
  posets. For simplicity, we treat $\algebra A$ as a subset of
  $\algebra B$. Suppose $S \subseteq \algebra B$ and that $\join S \in
  \cB$ exists. Then since $\cC$ is compatibly complete and
  $\theta(\join S)$ is an upper bound for $\theta[S]$, the join
  $\join\theta[S]$ exists, and $\theta(\join S) \geq \join
  \theta[S]$. For the reverse inequality, let $a \in \algebra A$ be
  bounded above by $\join S$. Then $a = a \bmeet \join \bigcup_{s \in
    S}\{a_1 \in \algebra A\mid a_1 \leq s\} = \join \bigcup_{s \in
    S}\{a \bmeet a_1 \in \algebra A\mid a_1 \leq s\}$. So the latter
  join exists not only in $\cB$ but also in $\algebra A$, since $a \in
  \cA$. Hence $\iota'(a) = \join \bigcup_{s \in S}\{\iota'(a) \bmeet
  \iota'(a_1) \in \algebra A\mid a_1 \leq s\}$, as $\iota'$ is a
  complete homomorphism. Thus $\join\theta[S] \geq \iota'(a)$. As
  $\theta(\join S)$ is by definition the join of these $\iota'(a)$, we
  are done.
  
  Now we show that $\theta$ preserves $-$.  Since $\theta$ is order
  preserving, it is clear that $\theta(b_1 \bmeet b_2) \leq
  \theta(b_1) \bmeet \theta(b_2)$. As $b_1$ is the join of $\{b_1
  \bmeet b_2, b_1 - b_2\}$, and $\theta$ preserves joins, we deduce
  that $\theta(b_1) \leq \theta(b_1)\bmeet \theta(b_2) +
  \theta(b_1-b_2)$. Thus, in the Boolean algebra
  $\theta(b_1)^\downarrow$, we have $\theta(b_1)-\theta(b_2) =
  \theta(b_1)-\theta(b_1)\bmeet \theta(b_2) \leq \theta(b_1-b_2)$.
  For the reverse inequality, suppose $a \in \algebra A$ with $a \leq
  b_1 -b_2$. Then using \Cref{l:distributive}\ref{item:bb} we have
  \begin{align*}\theta (b_1) - \theta (b_2) = \theta (b_1) -
    \join_{\substack{a_2 \in \algebra A\\ a_2 \leq b_2}} \iota' (a_2)
    &=\meet_{\substack{a_2 \in \algebra A\\ a_2 \leq b_2}} (\theta
    (b_1) - \iota' (a_2))\\
    &\geq \meet_{\substack{a_2 \in \algebra A\\ a_2 \leq b_2}} (\iota' (a) - \iota' (a_2)) =
    \meet_{\substack{a_2 \in \algebra A\\ a_2 \leq b_2}} \iota' (a- a_2).
 \end{align*}
 But when $a \leq b_1 - b_2$ and $a_2 \leq b_2$, we know $a - a_2 = a$
 (by an elementary property of sets/partial functions). Hence $\theta
 (b_1) - \theta (b_2) \geq \iota'(a)$. As $\theta(b_1 - b_2)$ is by
 definition the join of these $\iota'(a)$, we are finished.
  
 Let us show that $\theta$ preserves $\rest$. We now treat $\algebra
 A$ as a subset of both $\algebra B$ and $\algebra C$, conflating
 $\algebra A$ with its images under $\iota$ and $\iota'$. For $b_1,
 b_2 \in \algebra B$ define $A_1 = \{a_1 \in \algebra A \mid a_1 \leq
 b_1\}$ and $A_2 = \{a_2 \in \algebra A \mid a_2 \leq b_2\}$. We
 calculate
 \begin{align*}
   \theta(b_1) \rest \theta(b_2)
   &= \join_{\algebra C} A_1 \rest \join_{\algebra C} A_2 &&\text{by the definition of $\theta$} \\
   &= \join_{\algebra C}(A_1 \rest A_2) &&\text{by \Cref{l:distributive}\ref{item:aa}}\\
   &=\theta(\join_{\algebra B}(A_1 \rest A_2))&&\text{as $\theta$ is join complete}\\
   &=\theta(\join_{\algebra B} A_1 \rest \join_{\algebra B} A_2)&&\text{by
                                                                   \Cref{l:distributive}\ref{item:aa}}\\ 
   &= \theta(b_1 \rest b_2)&&\text{by density of $\algebra A$ in $\algebra B$}.
 \end{align*} 
 
 It remains to show that $\theta$ is injective. We claim that since
 $\theta$ is a homomorphism, injectivity of $\theta$ amounts to having
 $\theta^{-1}(0) = \{0\}$. Suppose $\theta^{-1}(0) = \{0\}$. Then
 given $a, b \in \algebra A$ we have $\theta(a) = \theta(b) \implies
 \theta(a-b) = 0 = \theta(b-a)$, and hence $a-b = 0 = b-a$. By an
 elementary property of sets/partial functions, this implies $a = b$,
 proving the claim.
 Now to show $\theta^{-1}(0) = \{0\}$, let $b \in \cB$ satisfy $\theta(b) =
 0$. By the definition of $\theta$, this means that $\iota'(a) = 0$
 whenever $\iota(a) \leq b$. But since $\iota'$ is injective, if
 $\iota'(a) = 0$ then $a = 0$ and so, $\iota(a) = 0$. Since
 $\iota[\cA]$ is join dense in~$\cB$, we may conclude that $b = 0$, as
 required.
\end{proof}

Now we can show that compatible completions are unique up to unique isomorphism.

\begin{proposition}\label{p:1_0}
  If $\iota \from \algebra A \hookrightarrow \algebra C$ and $\iota'
  \from \algebra A \hookrightarrow \algebra C'$ are compatible
  completions of the representable $\{-, \rest\}$-algebra $\cA$ then
  there is a unique isomorphism $\theta \from \algebra C \to \algebra
  C'$ satisfying the condition $\theta \circ \iota = \iota'$.
\end{proposition}

\begin{proof}
  For uniqueness, suppose we have an isomorphism $\theta \from
  \algebra C \to \algebra C'$ satisfying $\theta \circ \iota =
  \iota'$. As $\iota'[\algebra A]$ is both join dense in $\algebra C'$
  and a subset of $\theta[\algebra C]$, by applying
  \Cref{lem:dense_complete} to $\theta$ we see that $\theta$ is
  complete.  Then as each $c \in \algebra C$ is equal to $\join
  \{\iota(a)\mid a \in \algebra A\text{ and } \iota(a) \leq c\}$, our $\theta$ must be given by
  $\theta \from c \mapsto \join \{\iota'(a) \mid a \in \algebra A\text{ and } \iota(a) \leq c\}$.

For existence, we argue that $\theta \from c \mapsto \join \{\iota'(a) \mid a \in \algebra A\text{ and } \iota(a) \leq c\}$ indeed works. As compatible completions are complete homomorphisms, we can apply \Cref{lemma:key} to $\iota$ and $\iota'$. Hence $\theta$ is a well-defined complete embedding of $\{-, \rest\}$-algebras. By symmetry, there is a complete embedding $\theta' \from \algebra C' \hookrightarrow \algebra C$ with $\theta' \circ \iota' = \iota$. Then $\theta' \circ \theta$ is complete and fixes $\iota[\algebra A]$. So by join density of $\iota[\algebra A]$, the homomorphism $\theta' \circ \theta$ is the identity on $\algebra C$. Similarly, $\theta \circ \theta'$ is the identity on $\algebra C'$. Thus $\theta$ is an isomorphism.
\end{proof}

Hence if an algebra has a compatible completion then that compatible completion is `unique', for which reason we may refer to a compatible completion as \emph{the} compatible completion. We may also, as is common, refer to $\algebra C$ itself as the compatible completion of $\algebra A$, when $\iota \from \algebra A \hookrightarrow \algebra C$ is a compatible completion.

Next, we show how to explicitly construct the completion of any
\emph{atomic} representable algebra, by showing that the monad on
atomic representable algebras induced by the adjunction of
\Cref{t:adj} gives precisely the compatible completion of the
algebra. As a corollary, we obtain a duality for compatibly complete
atomic representable algebras. Recall from Section~\ref{sec:F} that,
for every atomic representable algebra~$\cA$, we use $\pi_{\algebra
  A}$ to denote the canonical projection $\At(\cA) \twoheadrightarrow
\At(\cA)/{\sim_\cA}$ (see also Definition~\ref{sec:equiv}).

\begin{theorem}\label{p:completion}
  For every atomic representable $\{-, \rest\}$-algebra $\cA$, the homomorphism
  \begin{align*}
    \eta_\cA \from \cA
    & \to (G \circ F) (\cA) = \{f \from \At(\cA)/{\sim_\cA}\parrow \At(\cA) \mid f
      \subseteq \pi_{\algebra A}^{-1}\}
    \\ a
    & \mapsto  \{([x], x) \mid x \in \At(\cA)
      \textrel{and} x \leq a\}
  \end{align*}
  is the compatible completion of $\cA$.
\end{theorem}
 
\begin{proof}
  For injectivity, suppose $\eta_\cA(a) = \eta_\cA(b)$. Then as
  $\algebra A$ is atomic, it is atomistic. (Recall \Cref{def:atomistic} and the following remark.) So we have
  \[a = \join \{x \in \At(\cA) \mid x \leq a\} = \join \{x \in
    \At(\cA) \mid x \leq b\} = b.\]
  For compatible completeness of $(G \circ F)(\algebra A)$, let $S$ be
  a pairwise-compatible subset of $(G \circ F)(\algebra A)$. Then as
  $(G \circ F)(\algebra A)$ is an algebra of partial functions, all
  pairs of elements of $S$ agree on their shared domains. That is,
  $\bigcup S$ is a partial function.  Given that $f \subseteq
  \pi_{\algebra A}^{-1}$ for each $f \in S$, we have $\bigcup S
  \subseteq \pi_{\algebra A}^{-1}$. So $\bigcup S \in (G \circ
  F)(\algebra A)$ and is the least upper bound of $S$.

For join density of $\eta_\cA[\algebra A]$, it suffices to note the join density of $\eta_\cA[\At(\algebra A)]$.
\end{proof}

As a  consequence we have the following corollaries.
\begin{corollary}\label{t:discrete-duality} There is a duality between $\typeface
  C\aralg$ and $\setq$, where $\typeface C\aralg$ is the full
  subcategory of $\aralg$ consisting of the compatibly complete
  algebras.
\end{corollary}

\begin{proof}
  Clearly $G$ maps every set quotient $\pi$ to a compatibly complete
  algebra.  Thus, $G$ co-restricts to a functor
  $\setq^{\operatorname{op}} \to \typeface C\aralg$.
  Moreover, as observed in the proof of~\Cref{t:adj}, the functor $F
  \circ G$ is naturally isomorphic to the identity functor on
  $\setq$. To conclude that $\typeface C\aralg$ and $\setq$ are dually
  equivalent, it only remains to argue that $\eta$ restricted to
  $\typeface C\aralg$ provides a natural isomorphism from
  $\operatorname{Id}_{\typeface C\aralg}$ to $G \circ F$. It suffices
  to show that $\eta_\cA$ is an isomorphism (algebraically speaking)
  for every compatibly complete $\cA$ (given that isomorphisms, and
  hence their inverses, are complete homomorphisms). For this we just
  note that if $\cA$ is compatibly complete then the identity map $\cA
  \hookrightarrow \cA$ is a compatible completion of $\cA$, and then
  apply \Cref{p:completion} and \Cref{p:1_0}.
\end{proof}

\begin{corollary}
The category $\typeface C\aralg$ is a reflective subcategory of $\aralg$.
\end{corollary}

\begin{proof}
We saw in the proof of \Cref{t:discrete-duality} that the restriction of $G \circ F$ to $\typeface C\aralg$ is naturally isomorphic to the identity. It is a direct consequence that $G \circ F$, viewed as a functor $\aralg \to \typeface C\aralg$, is left adjoint to the inclusion $\typeface C\aralg \to \aralg$.
\end{proof}

We have now achieved all our main objectives for this section. However, in order to better situate these results, it is worth being precise about which category our definition of compatible completion inhabits. Thus we will specify that \Cref{def:completion} defines a \emph{compatible completion in $\operatorname{\bf{RepAlg}}$}, where $\operatorname{\bf{RepAlg}}$ is the category of representable $\{-, \rest\}$-algebras with $\{-, \rest\}$-homomorphisms. So \Cref{p:1_0} says that compatible completions in this category are unique, and \Cref{p:completion} shows how to construct them for atomic algebras. If now we replace $\operatorname{\bf{RepAlg}}$ with the category $\operatorname{\bf{RepAlg}_\infty}$ of representable $\{-, \rest\}$-algebras with \emph{complete} $\{-, \rest\}$-homomorphisms, we obtain the following definition.

\begin{definition}\label{def:completion2}
  A \defn{compatible completion in} $\operatorname{\bf{RepAlg}_\infty}$ of a representable $\{-,
  \rest\}$-algebra $\cA$ is a complete embedding $\iota \from \algebra A
  \hookrightarrow \algebra C$ of $\{-, \rest\}$-algebras such that $\cC$ is representable and compatibly complete and
  $\iota[\cA]$ is join dense in~$\cC$.
\end{definition}

Given \Cref{lem:dense_complete} and the fact that isomorphisms are complete homomorphisms, we can also claim that \Cref{p:1_0} says that compatible completions in $\operatorname{\bf{RepAlg}_\infty}$ are unique, and that \Cref{p:completion} shows how to construct them for atomic algebras.

 For Boolean algebras,  several equivalent definition of completions are
possible~\cite{completion}. The same is partially true for compatible completions in $\operatorname{\bf{RepAlg}_\infty}$.

\begin{proposition}\label{equivalent}
  Let $\iota \from \algebra A \hookrightarrow \algebra C$ be a complete
  embedding of representable $\{-, \rest\}$-algebras. Consider the following statements about $\iota$.
  \begin{enumerate}[label = (\alph*)]
  \item\label{item:a} $\algebra C$ is compatibly complete, and the
    image of $\algebra A$ is join dense in $\algebra C$.
  \item\label{item:c} $\algebra C$ is the `smallest' extension of
    $\algebra A$ that is compatibly complete. That is, $\algebra C$ is
    compatibly complete, and for every other complete embedding
    $\kappa \from \cA \hookrightarrow \cB$ into a compatibly complete
    and representable $\{-, \rest\}$-algebra~$\cB$, there exists a
    complete embedding $\widehat \kappa \from \cC \hookrightarrow \cB$
    making the following diagram commute.
    \begin{center}
      \begin{tikzpicture}[node distance = 20mm]
        \node (A) at (0,0) {$\cA$}; \node[right of = A] (C) {$\cC$};
        \node[below of = C, yshift = 5mm] (B) {$\cB$};
        
        \draw[right hook->] (A) to node[above] {$\iota$} (C); \draw
        [right hook->] (A) to node[below]{$\kappa$} (B); \draw[right hook->,
        dashed] (C) to node[right] {$\widehat \kappa$} (B);
      \end{tikzpicture}
    \end{center}
  \item\label{item:d} $\algebra C$ is the `largest' extension of
    $\algebra A$ in which the image of $\algebra A$ is join
    dense. That is, $\iota[\algebra A]$ is join dense in $\algebra C$,
    and for every other complete embedding $\kappa \from \cA
    \hookrightarrow \cB$ into a representable $\{-,
    \rest\}$-algebra~$\cB$ in which the image of $\algebra A$ is join
    dense, there exists a complete embedding $\widehat \kappa \from
    \cB \hookrightarrow \cC$ making the following diagram commute.
    \begin{center}
      \begin{tikzpicture}[node distance = 20mm]
        \node (A) at (0,0) {$\cA$}; \node[right of = A] (C) {$\cB$};
        \node[below of = C, yshift = 5mm] (B) {$\cC$};
        
        \draw[right hook->] (A) to node[above] {$\kappa$} (C); \draw
        [right hook->] (A) to node[below]{$\iota$} (B); \draw[right hook->,
        dashed] (C) to node[right] {$\widehat \kappa$} (B);
      \end{tikzpicture}
    \end{center}
  \end{enumerate}
Then $\ref{item:a} \implies \ref{item:c}$, and $\ref{item:a} \implies \ref{item:d},$ and if $\algebra A$ has a completion then all three conditions are equivalent.
\end{proposition}
\begin{proof} 
  For \ref{item:a} $\implies$ \ref{item:c}, apply \Cref{lemma:key} to
  $\iota$ and $\kappa$, using join density of $\iota[\algebra A]$ in
  $\algebra C$. For \ref{item:a} $\implies$ \ref{item:d}, apply
  \Cref{lemma:key} to $\kappa$ and $\iota$, using compatible
  completeness of $\algebra C$.

  For the last part, suppose $\algebra A$ has a completion $\kappa
  \from \algebra A \hookrightarrow \algebra B$. If \ref{item:c} holds
  for $\algebra C$, then it can be applied to $\kappa \from \algebra A
  \hookrightarrow \algebra B$. Then $\widehat\kappa$ is surjective,
  since now $\kappa[\algebra A]$ is join dense in $\algebra B$ and
  $\widehat\kappa$ is complete. Hence the embedding $\widehat\kappa$
  is in fact an isomorphism, and so~\ref{item:a} holds.

  Similarly, if~\ref{item:d} holds for $\algebra C$, then apply it to
  $\kappa \from \algebra A \hookrightarrow \algebra B$ to obtain
  $\widehat\kappa \from \algebra B \hookrightarrow \algebra C$. Then
  as the image of $\algebra A$ is also join dense in $\algebra B$ we
  similarly have $\widehat\iota \from \algebra C \hookrightarrow
  \algebra B$ commuting with the embedding of $\algebra A$. Both
  compositions $\widehat\kappa \circ \widehat\iota$ and $\widehat\iota
  \circ \widehat\kappa$ are complete homomorphisms fixing the embedded
  copies of $\algebra A$ (which are join dense), hence both
  compositions are the identity. So $\widehat\kappa$ is an isomorphism
  and \ref{item:a} holds.
\end{proof}

In light of \Cref{equivalent}, it would be interesting to know which
algebras in $\operatorname{\bf{RepAlg}_\infty}$ (beyond the atomic
ones) have compatible completions and how to construct those
completions. We leave this as an open problem.

\begin{problem}
  Which representable $\{-, \rest\}$-algebras have a compatible
  completion in $\operatorname{\bf{RepAlg}_\infty}$? Describe a
  general method to construct these completions.
\end{problem}

A set of implications similar to those in \Cref{equivalent} is not
possible for compatible completions in $\operatorname{\bf RepAlg}$, as
the following example shows.

\begin{example}\label{ex:counterexample}
  We will show that \ref{item:a} $\centernot\implies$ \ref{item:c} in
  \Cref{equivalent}, if we drop the assumption of completeness of the
  homomorphisms. Let $2 \coloneqq \{0, 1\}$, and let $\algebra F$ be
  the $\{-, \rest\}$-algebra consisting of the following partial
  functions $\mathbb N \parrow 2$:
  \begin{itemize}
  \item those with finite domain,
  \item those such that the inverse image of\hspace{.5pt} $0$ is a
    cofinite set.
  \end{itemize}
  It is straightforward to check that $\algebra F$ is closed under $-$
  and $\rest$, so is indeed a $\{-, \rest\}$-algebra of partial
  functions. This algebra $\algebra F$ is atomic, and its compatible
  completion $\algebra F'$ consists of all partial functions $\mathbb
  N \parrow 2$. Let $\{ \emptyset, \operatorname{Id}\}$ be the
  cardinality 2 $\{-, \rest\}$-algebra of both partial endofunctions
  on some singleton set, and let $\algebra G = \algebra F' \times \{
  \emptyset, \operatorname{Id}\}$. Then $\algebra G$ is also
  compatibly complete (and evidently representable by partial
  functions). Define $\kappa \from \algebra F \hookrightarrow \algebra
  G$ by $\kappa \from f \mapsto (f,\emptyset)$ for $f$ with finite
  domain and $\kappa \from f \mapsto (f,\operatorname{Id})$ otherwise.

  There does not exist a homomorphism $\widehat \kappa \from \algebra
  F' \to \algebra G$ such that $\iota \circ \widehat \kappa =
  \kappa$. For suppose otherwise, and consider the constant functions
  $\overline 1 \from \mathbb N \to 2$ belonging to $\algebra F'$ and
  $\overline 0 \from \mathbb N \to 2$ belonging to both $\algebra F$
  and $\algebra F'$. Now
  \[\widehat \kappa(\overline 1)\bmeet (\overline 0,
    \operatorname{Id}) = \widehat \kappa(\overline 1) \bmeet
    \kappa(\overline 0)= \widehat \kappa(\overline 1) \bmeet \widehat
    \kappa(\overline 0) = \widehat \kappa (\overline 1 \bmeet
    \overline 0) = \widehat \kappa(\emptyset)= \kappa(\emptyset) =
    (\emptyset, \emptyset),\]
  and hence the second component of $\,\widehat \kappa(\overline 1)$
  must be $\emptyset$. However
\[\widehat\kappa(\overline 1) \rest (\overline 0, \operatorname{Id}) = \widehat\kappa(\overline 1) \rest \kappa(\overline 0) = \widehat\kappa(\overline 1) \rest \widehat\kappa(\overline 0) =\widehat\kappa(\overline 1 \rest \overline 0) = \widehat\kappa(\overline 0) = \kappa(\overline 0) = (\overline 0, \operatorname{Id}), \]
indicating that the second component of $\,\widehat \kappa (\overline 1)$ is $\operatorname{Id}$---a contradiction.
\end{example}

Note that the issue in \Cref{ex:counterexample} cannot be overcome by restricting to the full subcategory of algebras having joins for all \emph{finite} pairwise-compatible sets, as $\algebra F$ already satisfies this finite compatible completeness condition.

\section{Discrete duality for compatibly complete algebras with
  operators}\label{sec:operators}

In this section we extend the adjunction, completion, and duality
results of the previous two sections to results allowing the algebras
to be equipped with arbitrary additional completely additive operators
respecting the compatibility structure. Unless specified otherwise,
let $\algebra A$ be an atomic representable $\{-, \rest\}$-algebra.

First we introduce the class of operations we are interested in.

\begin{definition}\label{def:compatibility-preserving}
  Let $\Omega$ be an $n$-ary operation on $\algebra A$. Then $\Omega$
  is \defn{compatibility preserving} if whenever $a_i, a'_i$ are
  compatible, for all $i$, we have that $\Omega(a_1, \dots, a_n)$ and
  $\Omega(a'_1, \dots, a'_n)$ are compatible.

  The operation $\Omega$ is \defn{completely additive} if whenever the
  supremum $\join S$ exists, for $S \subseteq \algebra A$, we have
  \begin{equation}\label{completely-additive}
    \Omega(a_1, \dots, a_{i-1}, \allowbreak\join  S,
    a_{i+1}, \dots, a_n) = \join\Omega(a_1,
    \dots, a_{i-1}, S, a_{i+1}, \dots, a_n)
  \end{equation}
  for any $i$ and any $a_1, \dots, a_{i-1}, a_{i+1}, \dots, a_n \in
  \algebra A$.
\end{definition}

It is worth being aware that in the literature on algebras `with operators', the term `operator' is not merely a synonym for `operation' but means \emph{finitely additive operation}, that is, \eqref{completely-additive} holds for all finite (possibly empty) $S$.

The operations we will treat are those that are both compatibility preserving and completely additive. Hence the algebraic categories we consider in this section take the following form, for a functional signature $\sigma$ (disjoint from $\{-, \rest\}$).
\begin{definition}
The category $\aralg(\sigma)$ has
\begin{itemize}
\item
objects: algebras of the signature $\{-, \rest\} \cup \sigma$ whose $\{-, \rest\}$-reduct is atomic and representable, and such that the symbols of $\sigma$ are interpreted as compatibility preserving completely additive operations,
\item morphisms: complete homomorphisms of $(\{-, \rest\} \cup
  \sigma)$-algebras.
\end{itemize}
\end{definition}

For reference, we will briefly list concrete operations $\Omega$ on
partial functions according to whether or not they are compatibility
preserving and completely additive---by which we mean the completely
representable $\{-, \rest, \Omega\}$-algebras are a (necessarily full)
subcategory of $\aralg(\{\Omega\})$.\footnote{Note this is not the
  same as the more subtle matter of whether `representable + $\{-,
  \rest\}$-reduct is completely representable $\implies$ subcategory'
  (see, for example, the case of composition in~\cite{complete}).}
Verifying that such inclusions hold is a simple matter, achieved by
checking the operation is compatibility preserving and completely
additive for actual algebras of partial functions in which any joins
are given by unions.
 
	We can list the following compatibility preserving and completely additive operations from the literature: \emph{composition} (usually denoted $\compo$), the unary $\D$ (\emph{domain}), $\R$ (\emph{range}), and $\F$ (\emph{fixset}) operations (the identity function restricted, respectively, to the domain, range, and fixed points of the argument \cite{hirsch}), the constant $1$ (\emph{identity}), and the binary $\vartriangleleft$ (\emph{range restriction}). \emph{Intersection} is of course already expressible in our base signature. 
	
	The signature obtained by adding composition to $\{-, \rest\}$ has been studied by Schein and the representation class axiomatised under the name of \emph{difference semigroups} \cite{schein1992difference}. The signature obtained by adding composition, domain, and identity is term equivalent to $\{\compo, \bmeet, \A\}$, for which the representation class is axiomatised in \cite{DBLP:journals/ijac/JacksonS11} and the completely representable algebras axiomatised in \cite{complete}. Adding range to this last signature we obtain $\{\compo, \bmeet, \A, \R\}$, and the representation class is axiomatised in \cite{hirsch}. Axiomatising the completely representable algebras for $\{\compo, \bmeet, \A, \R\}$ is currently an open problem.
	
	Operations that fail to be compatibility preserving and completely additive usually do so because they are not even order-preserving. We can list: \emph{antidomain} (identity function restricted to complement of domain) and its range analogue \emph{antirange}, similar negative versions of $\rest$ and $\vartriangleleft$, which we might call \emph{antidomain restriction} and \emph{antirange restriction} (antidomain restriction is called \emph{minus} in \cite{BERENDSEN2010141}), \emph{override} (also known as \emph{preferential union}) and \emph{update} (see \cite{BERENDSEN2010141}), \emph{maximal iterate} (see \cite{jackson2021restriction}) and \emph{opposite} (converse restricted to points with a unique preimage \cite{finiterep}). \emph{Converse} is an operation that \emph{is} completely additive but fails to be compatibility preserving. An interesting future project would be to extend the results of this section in a way that encompasses some of these operations, in particular the several that are either order-preserving or order-reversing in each coordinate. (An analogous extension of the theory of Boolean algebras with operators can be found in \cite{GEHRKE2001345}.)

Starting with a compatibility preserving and completely additive
$n$-ary operation $\Omega$, we can define an $(n+1)$-ary relation
$R_\Omega$ on the atoms of $\algebra A$
by \begin{equation}\label{make_relation}R_\Omega x_1{\dots} x_{n+1}
  \iff \Omega(x_1, \dots, x_n) \geq x_{n+1}.\end{equation}

Next we introduce a class of relations on set quotients, which will turn out to be precisely the relations obtained from operations in the way just described.

\begin{definition}
  Take sets $X, X_0$, and a surjection $\pi \from X \twoheadrightarrow
  X_0$, and let $R$ be an $(n+1)$-ary relation on $X$. The
  \defn{compatibility relation} $C \subseteq X \times X$ is given by $x C y$
  if and only if  $\pi(x) = \pi(y) \implies x = y$. Then $R$ has
  the \defn{compatibility property} (with respect to~$\pi$) if given
  $x_1C x'_1, \dots ,\ x_nC x'_n$ and $Rx_1{\dots} x_{n+1}$ and
  $Rx'_1{\dots} x'_{n+1}$, we have $x_{n+1}C x'_{n+1}$.
\end{definition}
Observe that, in the case where $\pi$ is the canonical projection
$\pi_\cA \from \At(\cA) \twoheadrightarrow \At(\cA)/{\sim_\cA}$ for
some atomic representable algebra~$\cA$, the compatibility relation
$C$ coincides with the compatibility relation introduced
in~\Cref{algebra_compatibility} (restricted to atoms).

Given an $R$ satisfying the compatibility property, we can define an
$n$-ary operation $\Omega_R$ on the dual $\algebra A_\pi$ of $\pi
\from X \twoheadrightarrow X_0$ by conflating elements of $\algebra
A_\pi$ with their image, and then
setting
\begin{equation}\label{make_operation}\Omega_R(X_1, \dots, X_n)
  = R(X_1, \dots, X_n, \_),
\end{equation}
where $R(X_1, \dots, X_n, \_) \coloneqq \bigcup_{x_1\in X_1, \dots, x_n \in
  X_n} \{x_{n+1} \in X \mid Rx_1{\dots} x_{n+1}\}$.  Notice that a subset $X'
\subseteq X$ defines a partial function of $\cA_\pi$ exactly when it
contains at most one element of each fibre of $\pi$, that is, when $xC
x'$ for every $x, x' \in X'$. So suppose that $x_{n+1}, x'_{n+1} \in
\Omega_R(X_1, \dots, X_n)$. Then there exist $x_1\in X_1, \dots, x_n
\in X_n$ with $Rx_1{\dots} x_{n+1}$, and $x'_1\in X_1, \dots, x'_n \in
X_n$ with $Rx'_1{\dots} x'_{n+1}$. Since $X_1, \dots, X_n$ are images
of partial sections, we have $x_1Cx'_1, \dots ,x_nCx'_n$. Then by the
hypothesis that $R$ has the compatibility property, if $x_{n+1}$ and
$x'_{n+1}$ lie in the same fibre they are equal. Therefore, $\Omega_R$
is indeed a well-defined operation on~$\cA_\pi$.

Finally, we define conditions that morphisms of set quotients are
required to satisfy, when those set quotients are equipped with
additional relations.

\begin{definition}
  Take a partial function $\varphi \from X \rightharpoonup Y$ and
  $(n+1)$-ary relations $R_X$ and $R_Y$ on $X$ and $Y$. Then $\varphi$
  satisfies the \defn{reverse forth condition} if whenever
  $R_Xx_1{\dots} x_{n+1}$ and $\varphi(x_1), \dots, \varphi(x_n)$ are
  defined, then $\varphi(x_{n+1})$ is defined and
  $R_Y\varphi(x_1){\dots} \varphi(x_{n+1})$. The partial map $\varphi$
  satisfies the \defn{back condition} if whenever $\varphi(x_{n+1})$
  is defined and $R_Yy_1{\dots} y_n\varphi(x_{n+1})$, then there exist
  $x_1, \dots, x_n \in \dom(\varphi)$ such that $\varphi(x_1) = y_1,
  \dots, \varphi(x_n) = y_n$ and $R_Xx_1{\dots} x_{n+1}$.
\end{definition}

We are now ready to extend the adjunction $F \from \aralg \dashv \setq^{\operatorname{op}} \,\from\! G$ to algebras with operators. We fix a functional signature $\sigma$ (disjoint from $\{-, \rest\}$). We have already defined $\aralg(\sigma)$.

\begin{definition}
The category $\setq(\sigma)$ has
\begin{itemize}
\item objects: the objects of $\setq$ equipped with, for each $\Omega
  \in \sigma$, an $(n+1)$-ary relation $R_\Omega$ that has the
  compatibility property, where $n$ is the arity of $\Omega$,
\item
morphisms: morphisms of $\setq$ that satisfy the reverse forth condition and the back condition with respect to $R_\Omega$, for every $\Omega \in \sigma$.
\end{itemize}
\end{definition}

We are required to note at this point that both the reverse forth
condition and the back condition are preserved by composition of
partial maps (and are also satisfied by identity maps); hence
$\setq(\sigma)$ is indeed a category.

\begin{theorem}\label{thm:expansion}
There is an adjunction $F' \from \aralg(\sigma) \dashv \setq(\sigma)^{\operatorname{op}} \,\from\! G'$ that extends the adjunction $F \dashv G$ of \Cref{sec:duality} in the sense that the appropriate reducts of $F'(\algebra A)$ and $G'(\pi \from X \twoheadrightarrow X_0)$ equal $F(\algebra A)$ and $G(\pi \from X \twoheadrightarrow X_0)$, respectively.
\end{theorem}

The proof of this theorem takes up most of the remainder of the
paper. The reader is likely to have understood by now how to form $F'$
and $G'$. For $F'$: given an atomic representable $\{-,
\rest\}$-algebra $\algebra A$ equipped with compatibility preserving
completely additive operators indexed by $\sigma$, take $F(\algebra
A)$ and equip it with, for each operation $\Omega$ on $\algebra A$,
the relation $R_\Omega$ defined according to
\eqref{make_relation}. For $G'$: given a set quotient $\pi \from X
\twoheadrightarrow X_0$ equipped with relations, take $G(\pi \from X
\twoheadrightarrow X_0)$ and equip it with, for each relation $R$, the
operation $\Omega_R$ defined according to \eqref{make_operation}. The
proof consists therefore of establishing the following facts.

\begin{enumerate}
\item The $F'$ we wish to define is well-defined on objects. That is,
  the defined relations $R_\Omega$ have the compatibility property
  (\Cref{l:10}).

\item
$F'$ is well-defined on morphisms. That is, for a morphism $h$ in
$\aralg$ that preserves additional operations, the partial map $Fh$
satisfies the reverse forth and the back conditions with respect to
each pair of relations (\Cref{l:11}).

\item
$G'$ is well-defined on objects: each defined $\Omega_R$ is a
compatibility preserving and completely additive operation (\Cref{l:12}).

\item
$G'$ is well-defined on morphisms: given a morphism $H$ in $\setq$,
the defined operations are preserved by $G\varphi$ (\Cref{l:13}).

\item The unit and counit used in \Cref{t:adj} are still permitted
  families of morphisms. That is, for each algebra $\algebra A$, the
  map $\eta_{\algebra A}$ preserves the additional operations, and for
  each set quotient $\pi$, the map $\lambda_\pi$ satisfies the reverse
  forth condition and the back condition (\Cref{l:14}).
\end{enumerate}

\begin{lemma}\label{l:10}
If an operation $\Omega$ is compatibility preserving, then $R_\Omega$ has the compatibility property.
\end{lemma}

\begin{proof}
  Take $x_1Cx'_1, \dots ,x_nCx'_n$ and $R_\Omega x_1\dots x_{n+1}$ and
  $R_\Omega x'_1\dots x'_{n+1}$. Then $x_i, x'_i \in \cA$ are compatible, for each $i$. So since $\Omega$ is compatibility preserving, $\Omega(x_1, \dots, x_n)$ and $\Omega(x'_1,
  \dots, x'_n)$ are compatible. We have, by the hypotheses and the
  definition of $R_\Omega $, that $\Omega(x_1, \dots, x_n) \geq
  x_{n+1}$ and $\Omega(x'_1, \dots, x'_n) \geq x'_{n+1}$. Hence
  $x_{n+1}$ and $x'_{n+1}$ are compatible elements of $\cA$. Now
  $C$ is just the restriction of compatibility to pairs of atoms,
  hence $x_{n+1}Cx'_{n+1}$, as required.
\end{proof}

\begin{lemma}\label{l:11}
  
  Let $h \from \algebra A \to \algebra B$ be a complete homomorphism
  of atomic representable $\{-, \rest\}$-algebras, and let
  $\Omega^{\algebra A}$ and $\Omega^{\algebra B}$ be compatibility
  preserving completely additive $n$-ary operations on $\algebra A$
  and $\algebra B$ respectively. If $h$ validates
  \[h(\Omega^{\algebra A}(a_1, \dots, a_n)) = \Omega^{\algebra
      B}(h(a_1), \dots, h(a_n)),\]
  then $Fh$ satisfies the reverse forth and the back conditions with
  respect to $R_{\Omega^{\algebra B}}$ and $R_{\Omega^{\algebra A}}$.
\end{lemma}

\begin{proof}
  We write $R_\cA$ for $R_{\Omega^{\algebra A}}$, and we write $R_\cB$
  for $R_{\Omega^{\algebra B}}$. For the reverse forth condition,
  suppose $R_\cB y_1\dots y_{n+1}$ holds and that $Fh(y_1), \dots,
  Fh(y_n)$ are defined. Denote $Fh(y_1), \dots, Fh(y_n)$ by $x_1,
  \dots, x_n$ respectively. By the definition of $Fh$, we have $h(x_i)
  \geq y_i$ for each $i$. Then $h(\Omega^{\algebra A}(x_1, \dots,
  x_n)) = \Omega^{\algebra B}(h(x_1), \dots, h(x_n)) \geq
  \Omega^{\algebra B}(y_1, \dots, y_n)$, as $\Omega^{\algebra B}$ is
  order preserving, since it is completely additive. But
  $\Omega^{\algebra B}(y_1, \dots, y_n) \geq y_{n+1}$ by the
  hypothesis $R_\cB y_1\dots y_{n+1}$ and the definition of $R_\cB$.
  Since $h(\Omega^{\algebra A}(x_1, \dots, x_n)) \geq y_{n+1}$, we
  have that $Fh$ is defined at $y_{n+1}$ and $\Omega^\cA(x_1, \dots,
  x_n) \geq Fh(y_{n+1})$.  By the definition of $R_\cA$, the relation
  $R_\cA x_1{\dots} x_{n}Fh(y_{n+1})$ holds, and the reverse forth
  condition is established.

  For the back condition, suppose that $Fh(y_{n+1})$ is defined and
  that the relation $R_\cA x_1{\dots} x_nFh(y_{n+1})$ holds. Write
  $x_{n+1}$ for $Fh(y_{n+1})$. Then by the definition of $R_\cA $, the
  inequality $\Omega^{\algebra A}(x_1, \dots, x_n) \geq x_{n+1}$
  holds. Hence
  \[\Omega^{\algebra B}(h(x_1), \dots, h(x_n)) = h(\Omega^{\algebra
      A}(x_1, \dots, x_n)) \geq h(x_{n+1}) \geq y_{n+1}.\]
  As $\algebra B$ is atomic, it is atomistic. Hence by iterative
  application of the complete additivity of $\Omega^{\algebra B}$ to
  each argument, we find 
  \[\join \{\Omega^{\algebra B}(y_1, \dots, y_n) \mid y_1, \dots,
    y_n \in \At(\algebra B) : y_1 \leq h(x_1), \dots, y_n \leq
    h(x_n)\} \geq y_{n+1}.\]
  Since $y_{n+1}$ is an atom, there are therefore some $y_1, \dots,
  y_n \in \At(\algebra B)$ with $y_i \leq h(x_i)$ for each $i$, such
  that $\Omega^{\algebra B}(y_1, \dots, y_n) \geq y_{n+1}$. Then by
  the definitions, $Fh(y_i) = x_i$ for each $i$, and $R_\cB y_1{\dots}
  y_{n+1}$ holds; hence the back condition is established.
\end{proof}

\begin{lemma}\label{l:12}
  If a relation $R$ has the compatibility property, then $\Omega_R$ is
  compatibility preserving and completely additive.
\end{lemma}
\begin{proof}
  To see that $\Omega_R$ is compatibility preserving, let $X_i, X'_i
  \in \algebra A_\pi$ be compatible, for each $i$. Since $\Omega_R(X_1
  \cup X_1', \dots, X_n \cup X'_n)$ is (the image of) a well-defined
  partial section, of which $\Omega_R(X_1 , \dots, X_n )$ and
  $\Omega_R( X'_1, \dots, X'_n)$ are restrictions, $\Omega_R(X_1 ,
  \dots, X_n )$ and $\Omega_R( X'_1, \dots, X'_n)$ are compatible.

  To see that~$\Omega_R$ is completely additive, let $X_1, \dots, X_n
  \in \cA_\pi$ and $i \in \{1, \dots, n\}$, and suppose $\mathcal S$
  is a subset of~$\cA_\pi$ whose join~$\join \mathcal S$ exists. So
  $\join \mathcal S = \bigcup \mathcal S$. It is clear from the
  definition of $\Omega_R$ that
  \begin{align*}\Omega_R(X_1, \dots, X_{i-1}, \bigcup \mathcal S,
    X_{i+1}, \dots, X_n)&= \bigcup_{x_i \in \bigcup {\mathcal S}}
                          \Omega_R(X_1, \dots, X_{i-1}, x_i, X_{i+1},
                          \dots, X_n)\\&= \bigcup_{T \in \mathcal
    S}\bigcup_{x_i \in T} \Omega_R(X_1, \dots, X_{i-1}, x_i, X_{i+1},
    \dots, X_n)\\&= \bigcup_{T \in \mathcal S}\Omega_R(X_1, \dots,
    X_{i-1}, T, X_{i+1}, \dots, X_n)
  \end{align*}
  and hence provides a supremum for $\{\bigcup_{T \in \mathcal
    S}\Omega_R(X_1, \dots, X_{i-1}, T, X_{i+1}, \dots, X_n) \mid T \in
  \mathcal S\}$, as required.
\end{proof}

\begin{lemma}\label{l:13}
  Let $\varphi \from X \parrow Y$ define a morphism in $\setq$ from
  $(\pi \from X \twoheadrightarrow X_0)$ to $(\rho \from Y
  \twoheadrightarrow Y_0)$, and let $R_X$ and $R_Y$ be $(n+1)$-ary
  relations on $X$ and $Y$ respectively, both having the compatibility
  property. If $\varphi$ satisfies the reverse forth and the back
  conditions with respect to $R_X$ and $R_Y$, then the $n$-ary
  operations $\Omega_{R_X}$ and $\Omega_{R_Y}$ validate
  \[G\varphi(\Omega_{R_Y}(Y_1, \dots, Y_n)) =
    \Omega_{R_X}(G\varphi(Y_1), \dots, G\varphi(Y_n)).\]
\end{lemma}

\begin{proof}
  Recall that we are identifying elements of $G(\pi)$ and
  $G(\rho)$---partial sections---with their images, and note that
  according to this view, $G\varphi$ is given by $\varphi^{-1}$
  (inverse image). Let $Y_1, \dots, Y_n \in G(\rho)$.
 
  First we show that
  \[\varphi^{-1}(\Omega_{R_Y}(Y_1, \dots, Y_n)) \subseteq
    \Omega_{R_X}(\varphi^{-1}(Y_1), \dots, \varphi^{-1}(Y_n)),\]
  and later the reverse inclusion. So let $x_{n+1}$ be an arbitrary
  element of $\varphi^{-1}(\Omega_{R_Y}(Y_1, \dots, \allowbreak
  Y_n))$. That is, $\varphi$ is defined on $x_{n+1}$, with value
  $y_{n+1}$ say, and there are some ${y_1 \in Y_1}, \dots, {y_n \in
    Y_n}$ such that $Ry_1{\dots}y_{n+1}$. So by the back condition,
  there exist $x_1, \dots, x_n \in \dom(\varphi)$ such that
  $\varphi(x_1) = y_1, \dots, \varphi(x_n) = y_n$ and $R_Xx_1{\dots}
  x_{n+1}$. Then $x_i \in \varphi^{-1}(Y_i)$ for each $i \leq
  n$. Hence by the definition of $\Omega_{R_X}$, we have $x_{n+1} \in
  \Omega_{R_X}(\varphi^{-1}(Y_1), \dots, \varphi^{-1}(Y_n))$, as
  required.

  Now we show that
  \[\varphi^{-1}(\Omega_{R_Y}(Y_1, \dots, Y_n)) \supseteq \Omega_{R_X}(\varphi^{-1}(Y_1), \dots, \varphi^{-1}(Y_n)).\]
  Let $x_{n+1}$ be an arbitrary element of
  $\Omega_{R_X}(\varphi^{-1}(Y_1), \dots, \varphi^{-1}(Y_n))$, so
  there exist $x_1 \in \varphi^{-1}(Y_1), \dots, x_n \in
  \varphi^{-1}(Y_n)$ with $R_Xx_1{\dots}x_{n+1}$. (So in particular
  $\varphi(x_1), \dots, \varphi(x_n)$ are defined.) Then by the
  reverse forth condition, $\varphi(x_{n+1})$ is defined and
  $R_Y\varphi(x_1){\dots} \varphi(x_{n+1})$. Since $\varphi(x_i) \in
  Y_i$, for $i \leq n$, we then have $\varphi(x_{n+1}) \in
  \Omega_{R_Y}(Y_1, \dots, Y_n)$. Hence $x_{n+1} \in
  \varphi^{-1}(\Omega_{R_Y}(Y_1, \dots, Y_n))$.
\end{proof}

\begin{lemma}\label{l:14}
  Let $\algebra A$ be an atomic representable $\{-, \rest \}$-algebra
  and $\Omega$ be a compatibility preserving completely additive
  $n$-ary operation on $\algebra A$. Then the map $\eta_\algebra A$
  used in \Cref{t:adj} validates
  \[\eta_{\algebra A}(\Omega(a_1, \dots, a_n)) =
    \Omega_{R_\Omega}(\eta_{\algebra A}(a_1), \dots, \eta_{\algebra
      A}(a_n)).\]

Let $\pi \from X \twoheadrightarrow X_0$ be a set quotient, and $R$ be
an $(n+1)$-ary relation on $X$ with the compatibility property. The map $\lambda_\pi$ used in \Cref{t:adj} satisfies the reverse forth
condition and the back condition with respect to $R$ and
$R_{\Omega_R}$.
\end{lemma}

\begin{proof}
  For the first part, we unwrap the definition of the right-hand side
  of the equality.  Identifying partial functions with their images,
  we have
  \[\Omega_{R_\Omega} (\eta_{\algebra A}(a_1), \dots, \eta_{\algebra
      A}(a_n)) = R_{\Omega} (\At(a_1^\downarrow), \dots,
    \At(a_n^\downarrow),\_),\]
  and by the definition of $R_\Omega$, this set consists of all
  $x_{n+1} \in \At(\cA)$ for which there are $x_i \in
  \At(a_i^\downarrow)$, $i = 1, \dots, n$, such that $\Omega(x_1,
  \dots, x_n) \geq x_{n+1}$.  Since $\Omega$ is completely additive
  and $\algebra A$ is atomistic, such atoms $x_{n+1}$ are precisely
  those in the downset $\Omega(a_1, \dots, a_n)^\downarrow$. Thus
  \[\Omega_{R_\Omega} (\eta_{\algebra A}(a_1), \dots, \eta_{\algebra
      A}(a_n)) = \At(\Omega(a_1, \dots, a_n)^\downarrow),\]
  which is precisely the definition of the left-hand side
  $\eta_{\algebra A}(\Omega(a_1, \dots, a_n))$, so we are done.

  For the reverse forth and back conditions, since we saw in the proof
  of \Cref{t:adj} that $\pi$ and $(F \circ G)(\pi)$ are isomorphic via the
  correspondence $x \mapsto \{x\}$, it suffices to show that every
  relation $R \subseteq X^{n+1}$ coincides with $R_{\Omega_R}$ under
  this identification. And indeed, by definition, we have
  \[R_{\Omega_R}\{x_1\}\dots \{x_{n+1}\} \iff \Omega_R(\{x_1\}, \dots,
    \{x_{n}\}) \supseteq \{x_{n+1}\} \iff Rx_1 \dots x_{n+1}.\popQED\]
\end{proof}

This completes the proof that $F' \dashv G'$, and hence the proof of \Cref{thm:expansion}.

It is now straightforward to extend the completeness and duality
results of the previous section.

\begin{definition}\label{def:completion'}
Let $\algebra A$ be an algebra in $\aralg(\sigma)$.
  A \defn{compatible completion} of $\algebra A$  is an embedding $\iota \from \algebra A
  \hookrightarrow \algebra C$ of $(\{-, \rest\} \cup
  \sigma)$-algebras such that $\cC$ is in $\aralg(\sigma)$ and compatibly complete, and
  $\iota[\cA]$ is join dense in~$\cC$.
\end{definition}

\begin{corollary}\label{p:1_0'}
Let $\algebra A$ be an algebra in $\aralg(\sigma)$.
  If $\iota \from \algebra A \hookrightarrow \algebra C$ and $\iota'
  \from \algebra A \hookrightarrow \algebra C'$ are compatible completions of $\cA$ then there is a unique
  isomorphism $\theta \from \algebra C \to \algebra C'$ satisfying the
  condition $\theta \circ \iota = \iota'$.
\end{corollary}
\begin{proof}
Use the isomorphism $\theta$ from \Cref{p:1_0}. The fact that $\iota[\algebra A]$ is join dense in $\algebra C$ and the complete additivity of the additional operations ensure those additional operations are preserved by $\theta$.
\end{proof}

\begin{corollary}\label{p:completion'}
  For every algebra $\cA$ in $\aralg(\sigma)$, the embedding
  $\eta_\cA \from \cA \hookrightarrow (G' \circ F') (\cA)$ is the
  compatible completion of $\cA$.
\end{corollary}

\begin{corollary}\label{c:extended-duality}
  There is a duality between $\typeface C\aralg(\sigma)$ and
  $\setq(\sigma)^{\operatorname{op}}$, where $\typeface
  C\aralg(\sigma)$ is the full subcategory of $\aralg(\sigma)$
  consisting of the compatibly complete algebras.
\end{corollary}

\begin{proof}
  Given \Cref{t:discrete-duality}, we only need to check that the
  families of functions $\eta_{\algebra A}$ and $\lambda_\pi$ are
  still isomorphisms in the expanded categories, for which it only
  remains to show that the functions $\eta_{\algebra A}^{-1}$ and
  $\lambda_\pi^{-1}$ are valid morphisms in the expanded categories.
  
  We know $\eta_{\algebra A}$ is a bijection and preserves additional
  operations, and it is an elementary algebraic fact that this implies
  its inverse $\eta_{\algebra A}^{-1}$ preserves those same additional
  operations. Thus $\eta_{\algebra A}^{-1}$ is a morphism.
   
  For $\lambda_\pi^{-1}$, we must check that the reverse forth
  condition and the back condition are satisfied with respect to
  additional relations. But we saw in the proof of \Cref{l:14} that
  $\lambda_\pi$, and therefore $\lambda_\pi^{-1}$, preserves and
  reflects each additional relation. As $\lambda_\pi^{-1}$ is a
  bijection, it is then evident that the reverse forth and back
  conditions are respected.
\end{proof}

\begingroup
\setlength{\emergencystretch}{4pt}
\begin{corollary}
The category $\typeface C\aralg(\sigma)$ is a reflective subcategory of $\aralg(\sigma)$.
\end{corollary}
\endgroup

\bibliographystyle{amsplain}

\bibliography{../brettbib}


\end{document}